\documentclass[12pt]{amsart}
\usepackage{}

\usepackage{amsmath}

\usepackage{amsfonts}
\usepackage{amssymb}
\usepackage[all]{xy}           

\usepackage{bbding}
\usepackage{txfonts}
\usepackage{amscd}

\usepackage[shortlabels]{enumitem}
\usepackage{ifpdf}
\ifpdf
  \usepackage[colorlinks,final,backref=page,hyperindex]{hyperref}
\else
  \usepackage[colorlinks,final,backref=page,hyperindex,hypertex]{hyperref}
\fi
\usepackage{tikz}
\usepackage[active]{srcltx}

\makeatletter

\newtheorem{thm}{Theorem}[section]
\newtheorem{lem}[thm]{Lemma}
\newtheorem{cor}[thm]{Corollary}
\newtheorem{pro}[thm]{Proposition}
\newtheorem{ex}[thm]{Example}
\newtheorem{rmk}[thm]{Remark}
\newtheorem{defi}[thm]{Definition}

\newcommand {\emptycomment}[1]{}

\newcommand{\nc}{\newcommand}
\newcommand{\delete}[1]{}

\nc{\CV}{\mathbf{C}}

\newcommand{\CE}{\mathsf{CE}}
\newcommand{\LP}{\mathsf{LP}}

\newcommand{\ao}{\mathsf{AO}\mathsf{Lie}}
\newcommand{\rbo}{\mathsf{RBO}\mathsf{Lie}}
\newcommand{\rbol}{\mathsf{RBO}\mathsf{Leibniz}}
\newcommand{\rbola}{\mathsf{RBO}\mathsf{LeibnizA}}
\newcommand{\rbols}{\mathsf{RBO}\mathsf{LeibnizS}}
\newcommand{\lon }{\,\rightarrow\,}
\newcommand{\be }{\begin{equation}}
\newcommand{\ee }{\end{equation}}

\newcommand{\MC}{$\mathcal{MC}$}

\newcommand{\gr}{\theta}
\newcommand{\grl}{\bar{\varrho}^L}
\newcommand{\grr}{\bar{\varrho}^R}
\newcommand{\Lei}{\mathsf{Leib}}

\newcommand{\tn}{\mathfrak t}
\newcommand{\n}{\mathfrak n}

\newcommand{\g}{\mathfrak g}
\newcommand{\h}{\mathfrak h}
\newcommand{\la}{\mathfrak G} 

\newcommand{\huaB}{\mathcal{B}}

\newcommand{\huaR}{\mathcal{R}}

\newcommand{\huaF}{\mathcal{F}}
\newcommand{\huaG}{\mathcal{G}}

\newcommand{\huaHL}{\mathcal{HL}}

\newcommand{\huaP}{\mathcal{P}}
\newcommand{\huaC}{{\mathcal{C}}}

\newcommand{\huaI}{\mathcal{I}}

\newcommand{\huaH}{\mathcal{H}}

\newcommand{\huaO}{{\mathcal{O}}}

\newcommand{\huaZ}{\mathcal{Z}}

\newcommand{\frki}{\mathfrak i}

\newcommand{\frkC}{\mathfrak C}

\newcommand{\half}{\frac{1}{2}}

\newcommand{\Courant}[1]{\left\llbracket  #1\right\rrbracket }

\newcommand{\Id}{{\rm{Id}}}

\newcommand{\br}[1]{   [ \cdot,    \cdot  ]   }

\newcommand{\dM}{\mathrm{d}}

\newcommand{\Hom}{\mathrm{Hom}}

\newcommand{\Du}{\mathrm{Du}}
\newcommand{\BSu}{\mathrm{BSu}}

\newcommand{\Der}{\mathrm{Der}}
\newcommand{\Lie}{\mathrm{Lie}}

\newcommand{\gl}{\mathfrak {gl}}
\newcommand{\sln}{\mathfrak {sl}}     \newcommand{\dn}{\mathfrak {d}}

\newcommand{\kup}{Kupershmidt ~operator}

\nc{\oprn}{\theta}
\newcommand{\B}{\mathsf{B}}
\newcommand{\NR}{\mathsf{NR}}

\newcommand{\End}{\mathrm{End}}

\newcommand{\pr}{\mathrm{pr}}

\newcommand{\K}{\mathbb{K}}

\newcommand{\anti}{\mathrm{anti}}

\def\MC{\operatorname{MC}}

\newcommand{\rep}{\mathsf{Rep}}
\newcommand{\ARep}{\mathsf{ASRep}}
\newcommand{\SRep}{\mathsf{SRep}}

\newcommand{\GLA}{\mathsf{Gla}}
\newcommand{\RB}{\mathsf{C}}

\begin{document}

\title[A unified approach]{From relative Rota-Baxter operators and relative averaging
operators on Lie algebras to relative Rota-Baxter operators on
Leibniz algebras: a uniform approach
}

\author{Rong Tang}
\address{Department of Mathematics, Jilin University, Changchun 130012, Jilin, China}
\email{tangrong@jlu.edu.cn}

\author{Yunhe Sheng}
\address{Department of Mathematics, Jilin University, Changchun 130012, Jilin, China}
\email{shengyh@jlu.edu.cn}

\author{Friedrich Wagemann}
\address{Laboratoire de Mathematiques Jean Leray, University de Nantes, France }
\email{wagemann@math.univ-nantes.fr}


\begin{abstract}
In this paper, first   we construct two subcategories (using symmetric representations and antisymmetric representations) of the category of  relative Rota-Baxter operators on Leibniz algebras, and establish the relations with the categories  of  relative Rota-Baxter operators and  relative averaging  operators on Lie algebras. Then we show that there is a short exact sequence describing the relation between the controlling algebra of relative Rota-Baxter operators on a Leibniz algebra with respect to a symmetric (resp. antisymmetric) representation and the controlling algebra of the induced relative Rota-Baxter operators (resp. averaging operators) on the canonical Lie algebra associated to the Leibniz algebra. Finally,  we show that there is a long exact sequence   describing the relation between the cohomology groups of a relative Rota-Baxter operator  on a Leibniz algebra with respect to a symmetric (resp. antisymmetric) representation and the   cohomology groups of the induced relative Rota-Baxter operator (resp. averaging operator) on the canonical Lie algebra.

\end{abstract}



\maketitle

\tableofcontents

\allowdisplaybreaks


\section{Introduction}

The purpose of this paper is to provide a unified approach to study relative Rota-Baxter operators and relative averaging operators on Lie algebras.

\vspace{-1mm}
\subsection{Relative Rota-Baxter operators on Lie algebras}

G. Baxter introduced the concept of a Rota-Baxter commutative algebra  \cite{Ba} in his study of fluctuation theory in probability. A Rota-Baxter  operator of weight zero  on an associative algebra $(A,\cdot)$ is a linear map $\huaR:A\lon A$ such that
$$
(\huaR x)\cdot (\huaR y)=\huaR((\huaR x)\cdot y+ x\cdot (\huaR y)),\quad \forall x,y\in A.
$$
These operators with general weights have been found many applications in recent years, including  the algebraic approach of Connes-Kreimer \cite{CK} to renormalization of perturbative quantum field theory,  quantum analogue of Poisson geometry \cite{Uchino08},    twisting on associative algebras \cite{Uchino}, dendriform algebras and associative
Yang-Baxter equations \cite{Aguiar}. The notion of Rota-Baxter operators on Lie algebras was introduced independently in the 1980s as the operator form of the classical Yang-Baxter equation.  A linear operator $\huaR:\g\longrightarrow\g$ on a Lie algebra $\g$   is called a Rota-Baxter operator of weight zero if the following condition is satisfied:
$$
  [\huaR(x),\huaR(y)]_\g=\huaR([\huaR(x),y]_\g+[x,\huaR(y)]_\g),\quad\forall~x,y\in\g.
$$
Moreover, Kupershmidt  introduced the notion of a relative Rota-Baxter operator (also called an $\huaO$-operator) on a Lie algebra $\g$ with respect to  arbitrary representation in  \cite{Kuper1}.
  Note that 
  a skew-symmetric classical $r$-matrix is a relative Rota-Baxter operator on a Lie algebra with respect to the coadjoint representation. Relative Rota-Baxter operators play important roles in the study of integrable systems \cite{BGN2010,Semonov-Tian-Shansky}, provide solutions of the classical Yang-Baxter equation in the semidirect product Lie algebra and give rise to pre-Lie algebras \cite{Bai}.   See   \cite{Gub} for more details. Recently, the deformation and homotopy theory of  relative Rota-Baxter Lie algebras and relative Rota-Baxter associative algebras  were established in \cite{Das,LST,TBGS-1,TBGS-2}.

\subsection{Averaging operators on Lie algebras}
An averaging operator on an associative algebra $(A,\cdot)$ is a linear operator $\huaR:A\lon A$ such that
$$
(\huaR x)\cdot (\huaR y)=\huaR((\huaR x)\cdot y)=\huaR( x\cdot (\huaR y)),\quad \forall x,y\in A.
$$
In the last century, many studies on averaging operators were done for various special algebras, such as function spaces, Banach algebras, and the topics and methods were largely
analytic \cite{Barnett,Brainerd,Huijsmans,Rota}.  Recently, it was found that commutative double algebra structures  \cite{Goncharov} on a vector space $V$  can be described by symmetric averaging operators on the associative algebra $\End(V)$.  Averaging operators  can be defined on algebras over arbitrary binary operad  \cite{Aguiar}. In particular,
 a linear operator $\huaR:\g\longrightarrow\g$ on a Lie algebra $\g$   is called an averaging operator if the following condition is satisfied:
$$
  [\huaR(x),\huaR(y)]_\g=\huaR([\huaR(x),y]_\g ),\quad\forall~x,y\in\g.
$$
There is a close relation between the  classical Yang-Baxter equation, conformal algebras  and averaging operators on Lie algebras \cite{Kolesnikov}.

Averaging operators on Lie algebras are also called embedding tensors in some mathematical physics literatures.   Embedding tensors  and their associated tensor hierarchies provide an
 algebraic  and efficient way to construct supergravity theories and further to  construct  higher gauge theories (see e.g. \cite{BH, Hoh, KS, Str19}). Recently the controlling algebra and the cohomology theory of embedding tensors were established in \cite{STZ} using the higher derived bracket. Meanwhile the cohomology and homotopy theory of averaging associative algebras were studied in \cite{WZ}.

The concept of Rota-Baxter operators and averaging operators were further studied on the level of algebraic operads  in \cite{BBGN,PBG,Pei-Bai-Guo-Ni,PG,Uchino-2,Vallette}.  The action of the Rota-Baxter operator on
 a binary quadratic operad splits the operad, and the action of the averaging operator on a binary quadratic operad
duplicates the operad.  Since the splitting and duplication processes are in Koszul duality, one can regard the Rota-Baxter operator and averaging operator to be Koszul dual to each other. 

\subsection{Relative Rota-Baxter operators on Leibniz algebras}
Leibniz algebras were first discovered by Bloh who called them D-algebras  \cite{Bloh}. Then Loday rediscovered this algebraic structure and called them Leibniz algebras with the motivation in the study of the periodicity in algebraic K-theory \cite{Loday,Loday and Pirashvili,Loday-K}. Averaging operators on Lie algebras give rise to Leibniz algebras, which can be viewed as the duplication of Lie algebras. From the viewpoint of the operad theory,
leibniz algebras are defined as the algebras over the $Leibniz$ operad which is the  duplicator of the $Lie$ operad \cite{Pei-Bai-Guo-Ni}. More intrinsically, let $\huaP$ be a binary quadratic operad. The algebra duplicates by an averaging operator on the $\huaP$-algebra is a $\Du(\huaP)$-algebra. On the other hand, the algebra splits by a Rota-Baxter operator on the $\huaP$-algebra is a $\BSu(\huaP)$-algebra. Since $\Du(\huaP)$ is the Koszul dual of $\BSu(\huaP^!)$ \cite{Pei-Bai-Guo-Ni},
the duplication construction of an averaging operator  is a kind of Koszul dual of the splitting construction of a Rota-Baxter operator \cite{BBGN,Pei-Bai-Guo-Ni}. 
The (co)homology and homotopy theories of Leibniz algebras were established in \cite{FW,Loday and Pirashvili,Pirashvili,ammardefiLeibnizalgebra}.   Recently Leibniz algebras were studied from different aspects due to applications in both mathematics and physics.
In particular, integration of Leibniz algebras were studied in \cite{BW,Int1} and deformation quantization of Leibniz algebras was studied in \cite{DW}. As the underlying algebraic structures of embedding tensors and Courant algebroids, Leibniz algebras also have application in higher gauge theories \cite{BH,KS,SW} and homogeneous spaces \cite{Kinyon}. The notion of relative Rota-Baxter operators on Leibniz algebras was introduced in \cite{ST} to study Leibniz bialgebras.  Moreover, the cohomology theory of relative Rota-Baxter operators on Leibniz algebras was given in \cite{TSZ} and used to classify linear deformations and formal deformations.
\vspace{-2mm}

\subsection{Outline of the paper}

In this paper, we propose a unified approach to study relative Rota-Baxter operators and relative averaging operators on Lie algebras using relative Rota-Baxter operators on Leibniz algebras. First we observe that there are forgetful functors from the categories of relative Rota-Baxter operators and relative averaging operators on Lie algebras to the category of relative Rota-Baxter operators on Leibniz algebras. Conversely, we construct two functors from certain subcategories of the category of relative Rota-Baxter operators on Leibniz algebras to the categories of relative Rota-Baxter operators and relative averaging operators on Lie algebras, and show that they are left adjoint for the aforementioned forgetful functors. Then we show that the controlling algebras of relative Rota-Baxter operators   and relative averaging  operators on Lie algebras can be derived from the controlling algebra of relative Rota-Baxter operators    on Leibniz algebras. Finally, we give a long exact sequence to describe the relation between  the cohomology groups of a relative Rota-Baxter operator on a Leibniz algebra with respect to a symmetric (resp. antisymmetric) representation    and the cohomology groups of the induced relative Rota-Baxter operator (resp. averaging operator)    on  the canonical Lie algebra.  The concrete relations can be summarized by the following diagrams:
\begin{equation*}
\begin{split}
 \xymatrix{\text{$\rbols$} \ar[rr]^{\text{Theorem \ref{adjoint}}}\ar[d]_{\text{descendent}} &  & \text{$\rbo$}\ar[d]^{\text{descendent}}  \\
\text{Lie algebras}\ar[rr]^{ \text{\Id}} & &\text{Lie algebras}}
\,\, \xymatrix{\text{$\rbola$} \ar[rr]^{\text{Theorem \ref{AO-adjoint}}}\ar[d]_{\text{descendent}} &  & \text{$\ao$}\ar[d]^{\text{descendent}}  \\
\text{Leibniz algebras}\ar[rr]^{ \text{\Id}} & &\text{Leibniz algebras.}}
\end{split}
\end{equation*}
The paper is organized as follows. In Section \ref{sec:L}, using symmetric representations and antisymmetric representations, we construct two subcategories of the category of  relative Rota-Baxter operators on Leibniz algebras, and establish the relations with the categories  of  relative Rota-Baxter operators   and    relative averaging  operators on Lie algebras. In Section \ref{sec:control}, we show that there is a short exact sequence describing the relation between the controlling algebra of relative Rota-Baxter operators on a Leibniz algebra with respect to a symmetric representation and the controlling algebra of the induced relative Rota-Baxter operators on the canonical Lie algebra. For relative averaging operators, we establish a similar result. In Section \ref{sec:cohomology}, first we introduce the Loday-Pirashvili cohomology of a relative  Rota-Baxter operator on  a  Lie algebra and then use a long exact sequence to describe the relation between the cohomology groups of a relative Rota-Baxter operator  on a Leibniz algebra with respect to a symmetric representation and the Loday-Pirashvili cohomology groups of the induced relative Rota-Baxter operator  on the canonical Lie algebra. Similarly, there is a long exact sequence   describing the relation between the cohomology groups of a relative Rota-Baxter operator  on a Leibniz algebra with respect to an antisymmetric representation and the   cohomology groups of the induced relative averaging  operator  on the canonical Lie algebra.

In this paper, we work over an algebraically closed field $\K$ of characteristic 0 and all the vector spaces are over $\K$ and finite-dimensional.

\section{Relations between the categories $\ao$, $\rbo$ and  $\rbol$}\label{sec:L}

In this section, we establish the relations between certain subcategories of the category of  relative Rota-Baxter operators on Leibniz algebras and the category  of  relative Rota-Baxter operators on Lie algebras as well as the category of  relative averaging  operators on Lie algebras.

\subsection{Relative Rota-Baxter operators and relative averaging operators}

  \begin{defi}{\rm (\cite{Kuper1})}
 A  {\bf relative Rota-Baxter  operator}  on a  Lie algebra $(\g,[-,-]_\g)$ with respect to a representation $ (V;\rho) $ is a linear map $T:V\longrightarrow\g$ satisfying the following quadratic constraint:
 \begin{equation}  \label{eq:rbo}
   [Tu,Tv]_\g=T\big(\rho(Tu)(v)-\rho(Tv)(u)\big),\quad\forall u,v\in V.
 \end{equation}

  Let $T:V\longrightarrow\g$ (resp. $T':V'\longrightarrow\g'$) be a relative Rota-Baxter  operator on the Lie algebra $(\g,[-,-]_\g)$ (resp. $(\g',\{-,-\}_{\g'})$) with respect to the representation $(V;\rho)$ (resp. $(V',\rho')$).   A {\bf homomorphism} from  $ T$ to $ T' $ is a pair $(\phi,\varphi)$, where $\phi:\g\longrightarrow\g'$ is a Lie algebra homomorphism, $\varphi:V\longrightarrow V'$ is a linear map such that
         \begin{eqnarray}
          T'\circ \varphi&=&\phi\circ T,\label{defi:isocon1rb}\\
                \varphi\rho(x)(u)&=&\rho'(\phi(x))(\varphi(u)),\quad\forall x\in\g, u\in V.\label{defi:isocon2rb}
      \end{eqnarray}
   In particular, if $\phi$ and $\varphi$ are  invertible,  then $(\phi,\varphi)$ is called an  {\bf isomorphism}.
 \end{defi}
We denote by $\rbo$ the category of relative Rota-Baxter  operators on Lie algebras.

 \begin{defi}{\rm (\cite{Aguiar,Rota})}
 A  {\bf relative averaging operator}  on a  Lie algebra $(\g,[-,-]_\g)$ with respect to a representation $ (V;\rho) $ is a linear map $T:V\longrightarrow\g$ satisfying the following quadratic constraint:
 \begin{equation}\label{eq:aocon}
   [Tu,Tv]_\g=T\big(\rho(Tu)(v)\big),\quad\forall u,v\in V.
 \end{equation}

  Let $T:V\longrightarrow\g$ (resp. $T':V'\longrightarrow\g'$) be a relative averaging operator on the Lie algebra $(\g,[-,-]_\g)$ (resp. $(\g',\{-,-\}_{\g'})$) with respect to the representation $(V;\rho)$ (resp. $(V',\rho')$).   A {\bf homomorphism} from  $ T$ to $ T' $ is a pair $(\phi,\varphi)$, where $\phi:\g\longrightarrow\g'$ is a Lie algebra homomorphism, $\varphi:V\longrightarrow V'$ is a linear map  such that
         \begin{eqnarray}
          T'\circ \varphi&=&\phi\circ T,\label{defi:isocon1}\\
                \varphi\rho(x)(u)&=&\rho'(\phi(x))(\varphi(u)),\quad\forall x\in\g, u\in V.\label{defi:isocon2}
      \end{eqnarray}

      In particular, if $\phi$ and $\varphi$ are  invertible,  then $(\phi,\varphi)$ is called an  {\bf isomorphism}.
 \end{defi}

 We denote by $\ao$ the category of relative averaging operators on Lie algebras.

A (left) {\bf Leibniz algebra}  is a vector space $\la$ together with a bilinear operation $[-,-]_\la:\la\otimes\la\lon\la$ such that
\begin{eqnarray*}
\label{Leibniz}[x,[y,z]_\la]_\la=[[x,y]_\la,z]_\la+[y,[x,z]_\la]_\la,\quad\forall x,y,z\in\la.
\end{eqnarray*}

A {\bf representation} of a Leibniz algebra $(\la,[-,-]_{\la})$ is a triple $(V;\rho^L,\rho^R)$, where $V$ is a vector space, $\rho^L,\rho^R:\la\lon\gl(V)$ are linear maps such that  for all $x,y\in\la$,
\begin{eqnarray*}
\rho^L([x,y]_{\la})&=&[\rho^L(x),\rho^L(y)],\\
\rho^R([x,y]_{\la})&=&[\rho^L(x),\rho^R(y)],\\
\rho^R(y)\circ \rho^L(x)&=&-\rho^R(y)\circ \rho^R(x).
\end{eqnarray*}
Here $[-,-]$ is the commutator Lie bracket on $\gl(V)$. There are two important kinds of representations of Leibniz algebras:
 \begin{itemize}
   \item a representation $(V;\rho^L,\rho^R)$ is called {\bf symmetric} if $\rho^R=-\rho^L$;

   \item a representation $(V;\rho^L,\rho^R)$ is called {\bf antisymmetric} if $\rho^R=0$.
 \end{itemize}

For any representation $(V;\rho^L,\rho^R)$ of a Leibniz algebra $\la$, we define a vector subspace $V_{\rm anti}$ which is spanned by all elements  $\rho^L(x)v+\rho^R(x)v$ with $x\in\la$ and $v\in V$. Directly from $\rho^R(y)\circ \rho^L(x)=-\rho^R(y)\circ \rho^R(x)$ for all $x,y\in\la$, we deduce that $(V_{\rm anti};\rho^L,\rho^R)$ is an antisymmetric representation. The quotient representation $V/V_{\rm anti}$ is then symmetric, denoted by $V_{\rm sym}$, and we have an exact sequence
\begin{equation}   \label{short_exact_sequence}
0\to V_{\rm anti}\to V\to V_{\rm sym}\to 0.
\end{equation}

 Now we recall  the notion of a relative Rota-Baxter operator on a Leibniz algebra.
\begin{defi}{\rm (\cite{ST})}
Let $(V;\rho^L,\rho^R)$ be a representation of a Leibniz algebra $(\la,[-,-]_\la)$. A linear operator $T:V\lon\la$ is called a {\bf relative Rota-Baxter operator} on $(\la,[-,-]_\la)$ with respect to $(V;\rho^L,\rho^R)$ if $T$ satisfies:
\begin{eqnarray}\label{Rota-Baxter}
[Tu,Tv]_\la=T(\rho^L(Tu)v+\rho^R(Tv)u),\,\,\,\,\forall u,v\in V.
\end{eqnarray}

Let $T:V\longrightarrow\la$ (resp. $T':V'\longrightarrow\la'$) be a relative Rota-Baxter  operator on the Leibniz algebra $(\la,[-,-]_\la)$ (resp. $(\la',\{-,-\}_{\la'})$) with respect to the representation $(V;\rho^L,\rho^R)$ (resp. $(V;{\rho^L}',{\rho^R}')$).   A {\bf homomorphism} from  $T$ to $T'$ is a pair $(\phi,\varphi)$, where $\phi:\la\longrightarrow\la'$ is a Leibniz algebra homomorphism, $\varphi:V\longrightarrow V'$ is a linear map  such that for all $x\in\la, u\in V$,
         \begin{eqnarray}
          T'\circ \varphi&=&\phi\circ T,\label{defi:isocon1rbl}\\
                \varphi{\rho^L}(x)(u)&=&{\rho^L}'(\phi(x))(\varphi(u)), \label{defi:isocon2rbl}\\
                \varphi{\rho^R}(x)(u)&=&{\rho^R}'(\phi(x))(\varphi(u)). \label{defi:isocon3rbl}
      \end{eqnarray}
      In particular, if $\phi$ and $\varphi$ are  invertible,  then $(\phi,\varphi)$ is called an  {\bf isomorphism}  from $T$ to $T'$.
\end{defi}

We denote by $\rbol$ the category of relative Rota-Baxter  operators on Leibniz algebras.  Moreover, we denote by  $\rbols$ and  $\rbola$ the categories of relative Rota-Baxter  operators on Leibniz algebras with respect to symmetric representations and antisymmetric representations respectively.

\subsection{Relation to crossed modules}

Note that relative Rota-Baxter operators on Lie algebras and Leibniz algebras $T:V\to\g$ resp. $T:V\to\la$ as well as relative averaging operators on Lie algebras $T:V\to \g$ give rise to a bracket on $V$.

 For a relative Rota-Baxter operator $T:V\to\g$ on a Lie algebra $\g$, the bracket
$$[v,w]_T:=\rho(Tv)w-\rho(Tw)v$$
for all $v,w\in V$ renders $V$ a Lie algebra such that $T:V\to\g$ becomes a morphism of Lie algebras.

In the same vein, given a relative Rota-Baxter operator $T:V\to\la$ on a Leibniz algebra $\la$, the bracket
$$[v,w]_T:=\rho^L(Tv)w+\rho^R(Tw)v$$
for all $v,w\in V$ renders $V$ a Leibniz algebra such that $T:V\to\la$ becomes a morphism of Leibniz algebras.

Still in the same vein, given a relative averaging operator $T:V\to\g$ on a Lie algebra $\g$, the bracket
$$[v,w]_T:=\rho(Tv)w$$
for all $v,w\in V$ renders $V$ a Leibniz algebra such that $T:V\to\g$ is a morphism of Leibniz algebras.

The next proposition shows that with an equivariance condition, the action of $\g$ (resp. $\la$) on V becomes an action by derivations with respect to these brackets. This kind of equivariance condition is present, for example, in Loday-Pirashvili's tensor category of linear maps, see \cite{Loday-Pirashvili}.

\begin{pro}
\begin{enumerate}
\item Let $T:V\to\g$ be a relative Rota-Baxter operator on a Lie algebra and suppose that $T$ is equivariant in the sense that $T(\rho(x)v)=[x,Tv]_\g$ for all $x\in\g$ and all $v\in V$. We deduce that $\rho(x)$ for all $x\in\g$ is a derivation of the Lie algebra $(V,[-,-]_T)$. Thus, we obtain that $\rho:\g\lon\Der(V)$ is homomorphism of Lie algebras.
\item Let $T:V\to\la$ be a relative Rota-Baxter operator on a Leibniz algebra and suppose that $T$ is equivariant in the sense that $T(\rho^L(x)v)=[x,Tv]_\la$ and $T(\rho^R(x)v)=[Tv,x]_\la$ for all $x\in\g$ and all $v\in V$. Then $\la$ acts on $(V,[-,-]_T)$ by derivations (in the Leibniz sense), i.e. for all $x\in \la$ and all $v,w\in V$
\begin{align*}
&\rho^L(x)[v,w]_T = [\rho^L(x)v,w]_T+[v,\rho^L(x)w]_T,\\
&\rho^R(x)[v,w]_T =[v,\rho^R(x)w]_T-[w,\rho^R(x)v]_T,\\
&[\rho^L(x)v,w]_T +[\rho^R(x)v,w]_T=0.
\end{align*}

\end{enumerate}
\end{pro}
 We deduce that  the pair $(\rho^L(x),\rho^R(x))$ for all $x\in\la$ is a biderivation \cite{Loday} of the Leibniz algebra $(V,[-,-]_T)$.
For relative averaging operators, equivariance will imply that the left action is by derivations, but defining the right operation to be the opposite of the left action or zero will not render this an action by derivations in the Leibniz sense in general.

Note that this proposition does not give rise to crossed modules, because the requirements $\rho(Tv)w=[v,w]$ in the Lie case and $\rho^L(Tv)w=[v,w]=\rho^R(w)v$ in the Leibniz case for $v,w\in V$ are not fullfilled. In the case of relative averaging operators, this requirement is fullfilled, but the action is not by derivations in general.
On the other hand, it is well known that crossed modules of Lie algebras provide examples of relative averaging operators, see Example 2.14 in \cite{STZ}.

\subsection{Relations between the categories   $\rbo$ and  $\rbols$}

First we construct a functor $ F$ from the category $\rbo$ of relative Rota-Baxter operators on Lie algebras to  the category $\rbols$ of relative Rota-Baxter operators on Leibniz algebras with respect to  symmetric representations.

 On  objects, the functor $ F$ is defined as follows. Let $T:V\longrightarrow\g$   be a relative Rota-Baxter operator on the Lie algebra $(\g,[-,-]_\g)$   with respect to the representation $(V;\rho)$. We view the Lie algebra $(\g,[-,-]_\g)$ as a Leibniz algebra. Moreover the Lie algebra representation $\rho$ gives rise to a symmetric representation $(V;\rho,-\rho)$ of the Leibniz algebra $(\g,[-,-]_\g)$. Then the equality \eqref{eq:rbo} means that  $T:V\longrightarrow\g$ is a relative Rota-Baxter operator on the Leibniz algebra $(\g,[-,-]_\g)$ with respect to the  symmetric representation $(V;\rho,-\rho)$.

 On  morphisms, the functor $ F$ is defined as follows. Let $T:V\longrightarrow\g$ (resp. $T':V'\longrightarrow\g'$) be a relative Rota-Baxter operator on the Lie algebra $(\g,[-,-]_\g)$ (resp. $(\g',\{-,-\}_{\g'})$) with respect to the representation $(V;\rho)$ (resp. $(V',\rho')$).  Let $(\phi,\varphi)$ be a   homomorphism  from the relative Rota-Baxter operator $ T$ to $ T'$. Then it is obvious that $(\phi,\varphi)$ is also a homomorphism between the above induced relative Rota-Baxter operators on Leibniz algebras.

It is straightforward to see that $ F$ defined above is indeed a functor.

In the sequel, we construct a functor $ G$ from the category $\rbols$ to the category $\rbo$.

Let $(\la,[-,-]_\la)$ be a Leibniz algebra. Denote by $\Lei(\la)$ the ideal of squares spanned by all elements $[x,x]_\la$ for all $x\in\la$. We call $\Lei(\la)$ the {\bf  Leibniz kernel} of $(\la,[-,-]_\la)$. Observe that $\Lei(\la)=V_{\rm anti}$ for the adjoint representation $(V=\la,\rho^L,\rho^R)$ with $\rho^L(x)=[x,-]$ and $\rho^R(x)=[-,x]$. The following result is obvious.
\begin{lem}\label{canonical-Lie}
Let $(\la,[-,-]_\la)$ be a Leibniz algebra. Then
$$
\la_\Lie:=\la/\Lei(\la)
$$
is naturally a Lie algebra in which we denote the Lie bracket by $[-,-]_{\la_\Lie}$.
 \end{lem}
 We call $\la_\Lie$ the  {\bf  canonical Lie algebra} associated to the Leibniz algebra $(\la,[-,-]_\la)$.
 Observe that $\la_\Lie=V_{\rm sym}$ for the adjoint representation $(V=\la,\rho^L,\rho^R)$.
 The equivalence class of $x\in\la$ in $\la_\Lie$ will be denoted by $\bar{x}$. Moreover, $\Lei(\la)$ is contained in the left center of the Leibniz algebra $\la.$ We denote by
\begin{equation}
  \pr:\la\lon\la_\Lie
\end{equation}
the natural projection from the Leibniz algebra $\la$ to the canonical Lie algebra $\la_\Lie$. Obviously, the map $\pr$ preserves the bracket operation, i.e. the following equality holds:
\begin{equation}
  \pr[x,y]_\la=[\pr(x),\pr(y)]_{\la_\Lie},\quad \forall x,y\in \la.
\end{equation}

\emptycomment{
\begin{lem}\label{lem:imp}
  Let $(V;\rho^L,\rho^R)$ be a representation of a Leibniz algebra $(\la,[-,-]_{\la})$. Then we have
  $$
  \rho^L([x,x]_\la)=0,\quad \forall x\in\la.
  $$
\end{lem}

Consider the Lie algebra $\la_\Lie$. Define $\gr:\la_\Lie\longrightarrow \gl(V)$ by
\begin{eqnarray}
  \gr(\bar{x})=\rho^L(x),\quad \forall x\in\la.
\end{eqnarray}

By Lemma \ref{lem:imp}, the linear map $\gr$ is well-defined.

\begin{lem}\label{lem:rep}
   Let $(V;\rho^L,\rho^R)$ be a representation of a Leibniz algebra $\la$. Then $(V;\gr)$ is a representation of the canonical Lie algebra $\la_\Lie.$
\end{lem}
\begin{proof}For all $x,y\in\la$, we have
  \begin{eqnarray*}
    \theta[\bar{x},\bar{y}]_{\la_\Lie}=\theta\overline{[x,y]_\la}=\rho^L([x,y]_\la)=[\rho^L(x),\rho^L(y)]=[\theta(\bar{x}),\theta(\bar{y})],
  \end{eqnarray*}
  which implies that  $(V;\gr)$ is a representation of the Lie algebra $\la_\Lie.$
\end{proof}
}

Let $(V;\rho^L,\rho^R)$ be a representation of a Leibniz algebra $(\la,[-,-]_{\la})$. Then $\rho^L:\la\lon\gl(V)$ is a Leibniz algebra homomorphism. We deduce that
 \begin{eqnarray}\label{lem:imp}
    \rho^L(\Lei(\la))=0.
  \end{eqnarray}
   Thus, there is exactly one Lie algebra homomorphism $\gr:\la_\Lie\longrightarrow \gl(V)$ which is defined by
\begin{eqnarray}\label{induce-rep}
  \gr(\bar{x})=\rho^L(x),\quad \forall x\in\la,
\end{eqnarray}
such that following diagram  of Leibniz algebra homomorphisms commute:
$$
\xymatrix{
   \la\ar[rr]^{\pr }\ar[dr]_{\rho^L} & & \la_\Lie\ar[dl]^{\gr} \\
     & \gl(V).&
    }
$$
It implies that  $(V;\gr)$ is a representation of the canonical Lie algebra  $\la_\Lie.$

\begin{pro}\label{pro:indRB}
Let $T:V\lon\la$ be   a relative Rota-Baxter operator  on a Leibniz algebra $(\la,[-,-]_{\la})$  with respect to a symmetric representation $(V;\rho^L,\rho^R=-\rho^L)$. Then $\bar{T}:=\pr\circ T:V\lon\la_\Lie$ is a relative Rota-Baxter  operator on the canonical Lie algebra  $\la_\Lie$ with respect to the representation $(V;\gr)$.
\end{pro}
\begin{proof}
  Applying the projection $\pr:\la\lon\la_\Lie$ to the both sides of \eqref{Rota-Baxter}, we obtain
  \begin{eqnarray*}
    [\pr\circ T(u),\pr\circ T(v)]_{\la_\Lie}&=&\pr\circ T(\rho^L(Tu)v-\rho^L(Tv)u)\\
    &=&\pr\circ T(\gr(\pr\circ T(u))v-\gr(\pr\circ T(v))u),
  \end{eqnarray*}
  which implies that $\pr\circ T:V\lon\la_\Lie$ is  a relative Rota-Baxter operator on the canonical Lie algebra $\la_\Lie$ with respect to the representation $(V;\gr)$.
\end{proof}

\begin{pro}\label{pro:rbmtorbm}
  Let $T:V\longrightarrow\la$ (resp. $T':V'\longrightarrow\la'$) be a relative Rota-Baxter  operator on the Leibniz algebra $(\la,[-,-]_\la)$ (resp. $(\la',\{-,-\}_{\la'})$) with respect to the symmetric representation $(V;\rho^L,\rho^R=-\rho^L)$ (resp. $(V;{\rho^L}',{\rho^R}'=-{\rho^L}')$).  Let $(\phi,\varphi)$ be a   homomorphism of relative Rota-Baxter operators on Leibniz algebras   from  $ T $ to $ T' $. Then $(\bar{\phi},\varphi)$ is a   homomorphism  of relative  Rota-Baxter operators on Lie algebras from  $ \bar{T} $ to $ \bar{T'} $, where the Lie algebra homomorphism $\bar{\phi}:\la_\Lie\lon\la'_\Lie$ is defined by
  $$
  \bar{\phi}(\bar{x}):=\overline{\phi(x)},\quad \forall x\in\la.
  $$
\end{pro}

\begin{proof}
By $\phi(\Lei(\la))\subset \Lei(\la')$, we obtain that $\phi$ induces the Lie algebra homomorphism $\bar{\phi}:\la_\Lie\lon\la'_\Lie.$ Then it is straightforward to deduce that $(\bar{\phi},\varphi)$ is a   homomorphism  of relative averaging operators from  $\bar{T} $ to $\bar{T'} $.
\end{proof}

Now we are ready to give the main result in this subsection.
\begin{thm}\label{adjoint}
  There is a functor $G$ from the category $\rbols$ of relative Rota-Baxter operators on Leibniz algebras with respect to  symmetric representations to the category $\rbo$ of relative  Rota-Baxter operators on Lie algebras, such that $G$ is a left adjoint for $F$.
\end{thm}

\begin{proof}
  Let $T:V\lon\la$ be a relative Rota-Baxter operator on a Leibniz algebra $(\la,[-,-]_\la)$ with respect to a symmetric representation $(V;\rho^L,\rho^R=-\rho^L)$. By Proposition \ref{pro:indRB},  $\bar{T}:=\pr\circ T:V\lon\la_\Lie$ is a relative Rota-Baxter operator on the canonical Lie algebra $\la_\Lie$ with respect to the representation $(V;\gr)$. Thus, on objects, we define
  $$
 G(T):=\bar{T}.
  $$

  Let $(\phi,\varphi)$ be a   homomorphism of relative Rota-Baxter operators on Leibniz algebras   from  $ T$ to $ T' $. By Proposition \ref{pro:rbmtorbm}, $(\bar{\phi},\varphi)$ is a   homomorphism  of relative Rota-Baxter operators from  $ \bar{T} $ to $ \bar{T'} $. Thus, on  morphisms, we define
  $$
 G(\phi,\varphi):=(\bar{\phi},\varphi).
  $$
  Then it is straightforward to see that $ G$ is indeed a functor. Moreover, let $\g$ be a Lie algebra, we have $\Lei(\g)=0$. Thus, we obtain that $G\circ F=\Id_{\rbo}.$ Thus, we have the  identity natural transformation $\varepsilon$ (the counit of the adjunction)
$$
\varepsilon=\Id_{\Id_{\rbo}}:\Id_{\rbo}=G\circ F\lon \Id_{\rbo}.
$$
Moreover, for any relative Rota-Baxter operator $K:W\longrightarrow\h$ on a Lie algebra $\h$ with respect to a representation $(W;\rho)$, relative Rota-Baxter operator $T:V\longrightarrow\la$ on a Leibniz algebra $\la$ with respect to a symmetric representation $(V;\rho^L,\rho^R=-\rho^L)$ and relative Rota-Baxter operator homomorphism $(\chi,\xi)$ from $G(T)$ to $K$, we deduce that $\chi\circ \pr$ is a Leibniz algebra homomorphism from $\la$ to $\h$ and
\begin{eqnarray*}
&&K\circ \xi=(\chi\circ \pr)\circ T\\
&&\xi(\rho^L(x)u)\stackrel{\eqref{induce-rep}}{=}\xi(\gr(\bar{x})u)=\rho(\chi(\bar{x}))\xi(u)=\rho\big((\chi\circ \pr)(x)\big)\xi(u),~\forall x\in\la,u\in V.
\end{eqnarray*}
Thus, there is exactly one relative Rota-Baxter operator homomorphism $(\chi\circ\pr,\xi)$  from $T$ to $F(K)$ such that the following diagram commutes:
$$
\xymatrix{
   K=G F(K)\ar[rr]^{({\Id}_\h,{\Id}_W)} & & K \\
     & G(T)\ar[ul]^{G(\chi\circ \pr,\xi)}\ar[ur]_{(\chi,\xi)},&
    }
$$
which implies that $G$ is  a left adjoint for $F$. The proof is finished.
\end{proof}

\begin{rmk}  \label{rmk:almost_equivalence}
The functors $G:\rbols\to\rbo$ and $F:\rbo\to\rbols$ are induced by those establishing the equivalence of categories between modules over a Lie algebra $\g$ and symmetric Leibniz representations over $\g$, viewed as a Leibniz algebra. From this, one may deduce that the pair $(F,G)$ constitutes almost an adjoint equivalence. In fact, $F$ is fully faithful, because the counit of the adjunction $\epsilon:G\circ F\to\Id$ is a natural isomorphism for each object of $\rbo$. On the other hand, the unit of the adjunction $\eta:\Id\to F\circ G$ is not an isomorphism. For an object $T:V\to\la$ in $\rbols$, $(F\circ G)(T)$ corresponds to viewing $\pr\circ T:V\to\la_\Lie$ again as an object in $\rbols$. We have a natural relative Rota-Baxter operator homomorphism  $(\pr,\Id_V)$ from $T$ to $\pr\circ T$
$$
\xymatrix{ V \ar[d]^{\Id_V} \ar[r]^T & \la \ar[d]^\pr \\ V \ar[r]^{\pr\circ T} & \la_\Lie, }
$$
but it is not an isomorphism in general.

Note that one may add the data which arises from the passage from $\la_\Lie$ to $\la$ to the category $\rbo$ in order to make it an equivalence of categories. The data consists of the isomorphism class of an  antisymmetric representation of $\la_\Lie$ (corresponding to $\Lei(\la)$) and of the Loday-Pirashvili cohomology class of a $2$-cocycle characterizing the abelian extension
\begin{eqnarray}\label{characterizing-2-coho}
0\to \Lei(\la)\to \la\to \la_\Lie\to 0.
\end{eqnarray}
Given a relative Rota-Baxter operator $T:V\to\g$ together with a $\g$-module $A$ (viewed as an antisymmetric representation of the Leibniz algebra $\g$) and a $2$-cocycle $\alpha\in Z^2(\g,A)$, we obtain a relative Rota-Baxter operator  $T:V\to\la$ for the Leibniz algebra $\la$ constructed from $(\g,A,\alpha)$ with $A=\Lei(\la)$ and $V$ viewed as a symmetric representation of $\la$, because $\Lei(\la)$ acts trivially from the left and thus also trivially from the right as the representation $V$ is symmetric. This construction depends on the choice of the cocycle and the module $A$, but becomes natural when passing to cohomology and isomorphism classes.
\end{rmk}



\subsection{Relations between the categories $\ao$  and  $\rbola$}

First we construct a functor $\huaF$ from the category $\ao$ of relative averaging operators on Lie algebras to  the category $\rbola$ of relative Rota-Baxter operators on Leibniz algebras with respect to antisymmetric representations.

On  objects, the functor $\huaF$ is defined as follows. Let $T:V\longrightarrow\g$   be a relative averaging operator on a Lie algebra $(\g,[-,-]_\g)$   with respect to a representation $(V;\rho)$. We view the Lie algebra $(\g,[-,-]_\g)$ as a Leibniz algebra. Moreover the Lie algebra representation $\rho$ gives rise to an antisymmetric representation $(V;\rho,0)$ of the Leibniz algebra $(\g,[-,-]_\g)$. Then the equality \eqref{eq:aocon} means that  $T:V\longrightarrow\g$ is a relative Rota-Baxter operator on the Leibniz algebra $(\g,[-,-]_\g)$ with respect to the antisymmetric representation $(V;\rho,0)$.

On  morphisms, the functor $\huaF$ is defined as follows. Let $T:V\longrightarrow\g$ (resp. $T':V'\longrightarrow\g'$) be a relative averaging operator on a Lie algebra $(\g,[-,-]_\g)$ (resp. $(\g',\{-,-\}_{\g'})$) with respect to a representation $(V;\rho)$ (resp. $(V',\rho')$).  Let $(\phi,\varphi)$ be a   homomorphism  from the relative averaging operator $ T $ to $ T' $. Then it is obvious that $(\phi,\varphi)$ is also a homomorphism between the above induced relative Rota-Baxter operators.

It is straightforward to see that $\huaF$ defined above is indeed a functor.

In the sequel, we construct a functor $\huaG$ from the category $\rbola$ to the category $\ao$.

\begin{pro}\label{pro:indET}
  Let $T:V\lon\la$ be a relative Rota-Baxter operator on a Leibniz algebra $(\la,[-,-]_\la)$ with respect to an antisymmetric representation $(V;\rho^L,\rho^R=0)$. Then $\bar{T}:=\pr\circ T:V\lon\la_\Lie$ is a relative averaging operator on the Lie algebra $\la_\Lie$ with respect to the representation $(V;\gr)$.
\end{pro}
\begin{proof}
  Applying the projection $\pr:\la\lon\la_\Lie$ to the both sides of \eqref{Rota-Baxter}, we obtain
  \begin{eqnarray*}
    [\pr\circ T(u),\pr\circ T(v)]_{\la_\Lie}=\pr\circ T(\rho^L(Tu)v)\stackrel{\eqref{induce-rep}}{=}\pr\circ T(\gr(\pr\circ T(u))v),
  \end{eqnarray*}
  which implies that $\pr\circ T:V\lon\la_\Lie$ is an averaging operator on the Lie algebra $\la_\Lie$ with respect to the representation $(V;\gr)$.
\end{proof}

Similar to Proposition \ref{pro:rbmtorbm}, we have the following result.

\begin{pro}\label{pro:rbmtoam}
  Let $T:V\longrightarrow\la$ (resp. $T':V'\longrightarrow\la'$) be a relative Rota-Baxter  operator on the Leibniz algebra $(\la,[-,-]_\la)$ (resp. $(\la',\{-,-\}_{\la'})$) with respect to the antisymmetric representation $(V;\rho^L,\rho^R=0)$ (resp. $(V;{\rho^L}',{\rho^R}'=0)$).  Let $(\phi,\varphi)$ be a   homomorphism of relative Rota-Baxter operators on Leibniz algebras   from  $ T$ to $ T' $. Then $(\bar{\phi},\varphi)$ is a   homomorphism  of relative averaging operators from  $ \bar{T} $ to $ \bar{T'} $, where the Lie algebra homomorphism $\bar{\phi}:\la_\Lie\lon\la'_\Lie$ is defined by
  $$
  \bar{\phi}(\bar{x}):=\overline{\phi(x)},\quad \forall x\in\la.
  $$
\end{pro}

Now we are ready to give the main result in this subsection.
\begin{thm}\label{AO-adjoint}
  There is a functor $\huaG$ from    the category $\rbola$ of relative Rota-Baxter operators on Leibniz algebras with respect to antisymmetric representations to the category $\ao$ of relative averaging operators on Lie algebras, such that $\huaG$ is  a left adjoint
for  $\huaF$.
\end{thm}

\begin{proof}
  Let $T:V\lon\la$ be a relative Rota-Baxter operator on a Leibniz algebra $(\la,[-,-]_\la)$ with respect to an antisymmetric representation $(V;\rho^L,\rho^R=0)$. By Proposition \ref{pro:indET},  $\bar{T}:=\pr\circ T:V\lon\la_\Lie$ is a relative averaging operator on the Lie algebra $\la_\Lie$ with respect to the representation $(V;\gr)$. Thus, on objects, we define
  $$
  \huaG(T):=\bar{T}.
  $$

  Let $(\phi,\varphi)$ be a   homomorphism of relative Rota-Baxter operators on Leibniz algebras   from  $T$ to $T'$. By Proposition \ref{pro:rbmtoam}, $(\bar{\phi},\varphi)$ is a   homomorphism  of relative averaging operators from  $ \bar{T} $ to $ \bar{T'} $. Thus, on  morphisms, we define
  $$
  \huaG(\phi,\varphi):=(\bar{\phi},\varphi).
  $$
  Then it is straightforward to see that $\huaG$ is indeed a functor.

By the similar argument as Theorem \ref{adjoint}, we can show  that $\huaG$ is  a left adjoint for $\huaF$. We omit details.
\end{proof}

\begin{rmk}
Like in Remark \ref{rmk:almost_equivalence}, the pair of functors $\huaF$ and $\huaG$ form almost an adjoint equivalence.
\end{rmk}

\section{Relations between the controlling algebras}\label{sec:control}

In this section, we establish the relation between the controlling algebra  of  relative Rota-Baxter operators on a Leibniz algebra with respect to a symmetric (resp. antisymmetric) representation and the controlling algebra of  the   relative Rota-Baxter (resp. averaging) operators on the canonical Lie algebra.


A permutation $\sigma\in\mathbb S_n$ is called an {\em $(i,n-i)$-shuffle} if $\sigma(1)<\cdots <\sigma(i)$ and $\sigma(i+1)<\cdots <\sigma(n)$. If $i=0$ or $n$, we assume $\sigma=\Id$. The set of all $(i,n-i)$-shuffles will be denoted by $\mathbb S_{(i,n-i)}$. The notion of an $(i_1,\cdots,i_k)$-shuffle and the set $\mathbb S_{(i_1,\cdots,i_k)}$ are defined analogously.

A degree $1$ element $\theta\in\g^1$ is called a Maurer-Cartan element  of a differential graded Lie algebra $( \oplus_{k\in\mathbb Z}\g^k,[\cdot,\cdot],d)$ if it
satisfies the  Maurer-Cartan  equation:
$
d \theta+\half[\theta,\theta]=0$. The set of Maurer-Cartan elements in a dgla $\g$ will be denoted by $\MC(\g)$.

Let $\g$ be a vector space. We consider the graded vector space $C^*(\g,\g)=\oplus_{n\ge 1}C^n(\g,\g)=\oplus_{n\ge 1}\Hom(\otimes^n\g,\g)$. An element $P\in C^{p+1}(\g,\g)$ is defined to have degree $p$. 
The {\bf Balavoine bracket} on the graded vector space $C^*(\g,\g)$ is given by:
\begin{eqnarray}\label{leibniz-bracket}
[P,Q]_\B=P\bar{\circ}Q-(-1)^{pq}Q\bar{\circ}P,\,\,\,\,\forall P\in C^{p+1}(\g,\g),Q\in C^{q+1}(\g,\g),
\end{eqnarray}
where $P\bar{\circ}Q\in C^{p+q+1}(\g,\g)$ is defined by
\begin{eqnarray}
P\bar{\circ}Q=\sum_{k=1}^{p+1}P\circ_k Q,
\end{eqnarray}
and $\circ_k$ is defined by
\begin{eqnarray*}
 \nonumber&&(P\circ_kQ)(x_1,\cdots,x_{p+q+1})\\
&=&\sum_{\sigma\in\mathbb S_{(k-1,q)}}(-1)^{(k-1)q}(-1)^{\sigma}P(x_{\sigma(1)},\cdots,x_{\sigma(k-1)},Q(x_{\sigma(k)},\cdots,x_{\sigma(k+q-1)},x_{k+q}),x_{k+q+1},\cdots,x_{p+q+1}).
\end{eqnarray*}

It is well known that

\begin{thm}{\rm (\cite{Balavoine-1,Fialowski})}\label{leibniz-algebra-B}
With the above notations, $(C^*(\g,\g),[-,-]_{\B})$ is a graded Lie algebra. Its Maurer-Cartan elements (as a differential graded Lie algebra with zero differential) are precisely the Leibniz algebra structures on $\g$.
\end{thm}

\subsection{The controlling algebra of relative Rota-Baxter operators on a Leibniz algebra}
Let $(V;\rho_V^L,\rho_V^R)$ be  a representation of a Leibniz algebra $(\la,[-,-]_\la)$. Then there is a Leibniz algebra structure on $\la\oplus V$  given by
\begin{eqnarray}\label{semi-direct}
[x+u,y+v]_{\ltimes}=[x,y]_\la+\rho_V^L(x)v+\rho_V^R(y)u, \quad \forall x,y\in\g,~ u,v\in V.
\end{eqnarray}
This Leibniz algebra is called the semidirect product of $\la$ and $(V;\rho_V^L,\rho_V^R)$, and denoted by $\la\ltimes_{\rho_V^L,\rho_V^R}V.$ We denote the above semidirect product Leibniz multiplication by $\mu.$ Consider the graded vector space
$$C^*(V,\la):=\oplus_{n\ge 1}C^n(V,\la)=\oplus_{n\geq 1}\Hom(\otimes^{n}V,\la),$$
where an element $g\in C^n(V,\la)$ is defined to be of degree $n$.

\begin{thm}\label{twilled-DGLA}{\rm (\cite{ST})}
With the above notations,  $(C^*(V,\la),\{-,-\}_V)$ is a   graded Lie algebra, where the graded Lie bracket $\{-,-\}_V:C^m(V,\la)\times C^n(V,\la)\lon C^{m+n}(V,\la)$ is given by the derived bracket as following:
\begin{eqnarray*}
\{g_1,g_2\}_V&=&(-1)^{m-1}[[\mu,g_1]_\B,g_2]_\B,\quad \forall g_1\in C^m(V,\la),~g_2\in C^n(V,\la).
\end{eqnarray*}
 More precisely, we have
{\footnotesize
\begin{eqnarray}
&&\nonumber\{g_1,g_2\}_V(v_1,v_2,\cdots,v_{m+n})\\
&=&\nonumber\sum_{k=1}^{m}\sum_{\sigma\in\mathbb S_{(k-1,n)}}(-1)^{(k-1)n+1}(-1)^{\sigma}g_1(v_{\sigma(1)},\cdots,v_{\sigma(k-1)},\rho_V^L(g_2(v_{\sigma(k)},\cdots,v_{\sigma(k+n-1)}))v_{k+n},v_{k+n+1},\cdots,v_{m+n})\\
&&\nonumber+\sum_{k=2}^{m+1}\sum_{\sigma\in\mathbb S_{(k-2,n,1)}\atop \sigma(k+n-2)=k+n-1}(-1)^{kn}(-1)^{\sigma}
g_1(v_{\sigma(1)},\cdots,v_{\sigma(k-2)},\rho_V^R(g_2(v_{\sigma(k-1)},\cdots,v_{\sigma(k+n-2)}))v_{\sigma(k+n-1)},v_{k+n},\cdots,v_{m+n})\\
&&\nonumber+\sum_{k=1}^{m}\sum_{\sigma\in\mathbb S_{(k-1,n-1)}}(-1)^{(k-1)n}(-1)^{\sigma}[g_2(v_{\sigma(k)},\cdots,v_{\sigma(k+n-2)},v_{k+n-1}),g_1(v_{\sigma(1)},\cdots,v_{\sigma(k-1)},v_{k+n},\cdots,v_{m+n})]_{\la}\\
&&\nonumber+\sum_{\sigma\in\mathbb S_{(m,n-1)}}(-1)^{mn+1}(-1)^{\sigma}[g_1(v_{\sigma(1)},\cdots,v_{\sigma(m)}),g_2(v_{\sigma(m+1)},\cdots,v_{\sigma(m+n-1)},v_{m+n})]_{\la}\\
&&\nonumber+\sum_{k=1}^{n}\sum_{\sigma\in\mathbb S_{(k-1,m)}}(-1)^{m(k+n-1)}(-1)^{\sigma}g_2(v_{\sigma(1)},\cdots,v_{\sigma(k-1)},\rho_V^L(g_1(v_{\sigma(k)},\cdots,v_{\sigma(k+m-1)}))v_{k+m},v_{k+m+1},\cdots,v_{m+n})\\
&&\nonumber+\sum_{k=1}^{n}\sum_{\sigma\in\mathbb S_{(k-1,m,1)}\atop\sigma(k+m-1)=k+m}(-1)^{m(k+n-1)+1}(-1)^{\sigma}
g_2(v_{\sigma(1)},\cdots,v_{\sigma(k-1)},\rho_V^R(g_1(v_{\sigma(k)},\cdots,v_{\sigma(k-1+m)}))v_{\sigma(k+m)},v_{k+m+1},\cdots,v_{m+n}).
\end{eqnarray}
}

Moreover, its Maurer-Cartan elements (where the differential is taken to be zero again) are relative Rota-Baxter operators on the Leibniz algebra $(\la,[-,-]_\la)$ with respect to the representation $(V;\rho_V^L,\rho_V^R)$.
\end{thm}

In fact, this graded Lie algebra construction is functorial. More precisely, we consider the category $\la$-$\rep$ of   representations of the Leibniz algebra $(\la,[-,-]_\la)$. Then, we have a contravariant functor $\RB_\la$ from the category $\la$-$\rep$ to the category $\GLA$ of graded Lie algebras.

Let $(W;\rho_W^L,\rho_W^R)$ and $(V;\rho_V^L,\rho_V^R)$ be two representations of the Leibniz algebra $(\la,[-,-]_\la)$.   A {\bf homomorphism} from  $(W;\rho_W^L,\rho_W^R)$ to $(V;\rho_V^L,\rho_V^R)$ is a linear map $\phi:W\lon V$ such that
         \begin{eqnarray}
          \phi(\rho_W^L(x)w)&=&\rho_V^L(x)\phi(w),\label{defi:rep-homo1}\\
                \phi(\rho_W^R(x)w)&=&\rho_V^R(x)\phi(w),\quad\forall x\in\la, w\in W.\label{defi:rep-homo2}
      \end{eqnarray}

Let $\phi:W\lon V$ be a homomorphism from the representation $(W;\rho_W^L,\rho_W^R)$ to $(V;\rho_V^L,\rho_V^R)$. Define a linear map $\Phi:\Hom(\otimes^nV,\la)\lon\Hom(\otimes^{n}W,\la), n\geq 1,$   by
\begin{eqnarray}\label{eq:phi}
\Phi(f):=f\circ \phi^{\otimes n},\,\,\,\,\forall
f\in\Hom(\otimes^nV,\la).
\label{eq:defiphi}
\end{eqnarray}

\begin{pro}\label{morphism}
With the above notations, $\Phi$ is a homomorphism from the graded Lie algebra $(C^*(V,\la),\{-,-\}_V)$ to $(C^*(W,\la),\{-,-\}_W)$.
\end{pro}

\begin{proof}
It follows directly from \eqref{defi:rep-homo1} and \eqref{defi:rep-homo2}.  We omit the details.
\end{proof}

Moreover, we have the following theorem.

\begin{thm}\label{functor}  Theorem \ref{twilled-DGLA} and Proposition
  \ref{morphism} give us a contravariant functor $\RB_\la$ from the category $\la$-$\rep$ to the category $\GLA$ of graded Lie algebras,
\begin{itemize}
  \item on   objects, the functor $\RB_\la:$ $\la$-$\rep$ $\to$ $\GLA$  is defined  by
\begin{eqnarray}
\RB_\la\Big((V;\rho_V^L,\rho_V^R)\Big)&=&(C^*(V,\la),\{-,-\}_V),
\end{eqnarray}
 \item on  morphisms, the functor $\RB_\la:$ $\la$-$\rep$ $\to$ $\GLA$ is defined by
 \begin{eqnarray}
\RB_\la(W\stackrel{\phi}{\lon}V)&=&(C^*(V,\la),\{-,-\}_V)\stackrel{\Phi}{\lon}(C^*(W,\la),\{-,-\}_W),
\end{eqnarray}
 \end{itemize}
 where  $(W;\rho_W^L,\rho_W^R)$ and $(V;\rho_V^L,\rho_V^R)$ are   representations of the Leibniz algebra $(\la,[-,-]_\la)$  and $\phi\in\Hom_{\la\mbox{-}\rep}(W,V)$.
\end{thm}

\begin{proof}
Let $(V;\rho_V^L,\rho_V^R)$ be a representation of the Leibniz algebra $(\la,[-,-]_\la)$. For the identity homomorphism $\Id_V:V\lon V$, we have $\RB_\la(\Id_V)=\Id_{C^*(V,\la)}$. Moreover, let $\phi:W\lon V$ and $\psi:V\lon U$ be two homomorphisms of representations of the Leibniz algebra $(\la,[-,-]_\la)$. For $n\geq 1$ and $\theta\in\Hom(U^{\otimes n},\la)$, we have
\begin{eqnarray*}
\RB_\la(\psi\circ\phi)\theta&\stackrel{\eqref{eq:defiphi}}{=}&\theta\circ (\psi\circ\phi)^{\otimes n}\\
&=&(\theta\circ \psi^{\otimes n})\circ \phi^{\otimes n}\\
&=&(\RB_\la(\phi)\circ\RB_\la(\psi))\theta.
\end{eqnarray*}
Thus, we deduce that $\RB_\la$ is a contravariant functor.
\end{proof}

As a consequence, the short exact sequence (\ref{short_exact_sequence})
$$0\to V_{\rm anti}\stackrel{\phi}{\to} V\stackrel{\psi}{\to} V_{\rm sym}\to 0$$
associated to the $\la$-representation $(V;\rho_V^L,\rho_V^R)$ induces homomorphisms of graded Lie algebras
$$(C^*(V_{\rm sym},\la),\{-,-\}_{V_{\rm sym}})\stackrel{\Psi}{\to} (C^*(V,\la),\{-,-\}_V)\stackrel{\Phi}{\to}(C^*(V_{\rm anti},\la),\{-,-\}_{V_{\rm anti}}),$$
which do not necessarily form an exact sequence of graded Lie algebras. By passage to the corresponding Maurer-Cartan sets, we obtain a sequence of induced maps
\begin{equation}   \label{Maurer-Cartan_sequence}
\MC\big(C^*(V_{\rm sym},\la)\big)\stackrel{\Psi}{\to} \MC\big(C^*(V,\la)\big)\stackrel{\Phi}{\to}\MC\big(C^*(V_{\rm anti},\la)\big),
\end{equation}
which describe (loosely) how a relative Rota-Baxter operator on the Leibniz algebra $(\la,[-,-]_\la)$ with respect to $(V;\rho_V^L,\rho_V^R)$ is assembled from a relative Rota-Baxter operator on the Lie algebra $\la_\Lie$ and a relative averaging operator on $\la_\Lie$, according to Theorems \ref{adjoint} and \ref{AO-adjoint}, see also Corollaries \ref{induce-gla-rota-baxter} and \ref{induce-gla-anti}.


The (not necessarily exact) sequence (\ref{Maurer-Cartan_sequence}) is the effect in the {\it module-variable} of the standard short exact sequence. In order to investigate also the effect in the {\it Leibniz-algebra-variable} of the standard short exact sequence \eqref{characteristic-element} (which will be the subject of Theorem \ref{thm:carb}), we need some more preparation.

\subsection{The controlling algebra of relative Rota-Baxter operators on a Lie algebra}

Associated to any Leibniz algebra $(\la,[-,-]_\la)$,   we have an exact sequence of Leibniz algebras
\begin{eqnarray}\label{characteristic-element}
0\longrightarrow \Lei(\la)\stackrel{\frki}{\longrightarrow}\la\stackrel{ \pr}{\longrightarrow} \la_\Lie\longrightarrow 0.
\end{eqnarray}
It is an abelian extension of $\la_\Lie$ by $\Lei(\la)$. Moreover the induced  representation of $\la_\Lie$ on $\Lei(\la)$ is an antisymmetric representation. Observe that the sequence (\ref{characteristic-element}) is the special case of the sequence (\ref{short_exact_sequence}) for the adjoint representation of $\la$.

Let $(V;\rho_V^L,\rho_V^R)$ be a  representation of a Leibniz algebra $(\la,[-,-]_\la)$. Thus, there is exactly one Lie algebra homomorphism $\gr_V:\la_\Lie\longrightarrow \gl(V)$, such that following diagram  of Leibniz algebra homomorphisms commutes:
\begin{eqnarray}
\xymatrix{
   \Lei(\la)\ar[rr]^{\frki}\ar[drr]_{\rho_V^L\circ\frki}  & & \la\ar[d]^{\rho_V^L} \ar[rr]^{\pr}& &\la_\Lie\ar[dll]^{\gr_V}\\
                               &&\gl(V).
    }
\end{eqnarray}

Consider the category $\la$-$\SRep$ of   symmetric representations of the Leibniz algebra $(\la,[-,-]_\la)$, which is a subcategory of $\la$-$\rep$. Then we have the following functors
\begin{eqnarray*}
&&\huaI_s:\la\mbox{-}\SRep\lon  \Lei(\la)\mbox{-}\SRep\\
&&\huaP_s:\la\mbox{-}\SRep\lon  \la_\Lie\mbox{-}\SRep
\end{eqnarray*}
which are given by
\begin{itemize}
  \item the functor $\huaI_s:$ $\la\mbox{-}\SRep$ $\to$ $\Lei(\la)\mbox{-}\SRep$, which is defined  on   objects and on   morphisms respectively by
\begin{eqnarray}
\huaI_s\Big((V;\rho_V^L,-\rho_V^L)\Big)&=&(V;\rho_V^L\circ\frki,-\rho_V^L\circ\frki),\\
\huaI_s(W\stackrel{\phi}{\lon}V)&=&W\stackrel{\phi}{\lon}V,
\end{eqnarray}
 \item the functor $\huaP_s:$ $\la\mbox{-}\SRep$ $\to$ $\la_\Lie\mbox{-}\SRep$, which is defined  on   objects and on   morphisms respectively by
\begin{eqnarray}
\huaP_s\Big((V;\rho_V^L,-\rho_V^L)\Big)&=&(V;\gr_V,-\gr_V),\\
\huaP_s(W\stackrel{\phi}{\lon}V)&=&W\stackrel{\phi}{\lon}V,
\end{eqnarray}
 \end{itemize}
for symmetric representations $(W;\rho_W^L,-\rho_W^L)$ and $(V;\rho_V^L,-\rho_V^L)$  of the Leibniz algebra $(\la,[-,-]_\la)$  and $\phi\in\Hom_{\la\mbox{-}\SRep}(W,V)$.

Thus, we have three functors $\RB_{\Lei(\la)}\circ\huaI_s,~\RB_\la$ and $\RB_{\la_\Lie}\circ\huaP_s$ from $\la\mbox{-}\SRep$ to $\GLA$, where the functor $\RB_\la$ is given in Theorem \ref{functor}. Moreover, for any symmetric representation $(V;\rho_V^L,-\rho_V^L)$ of the Leibniz algebra $(\la,[-,-]_\la)$, we define $$\alpha_V:(\RB_{\Lei(\la)}\circ\huaI_s)(V;\rho_V^L,-\rho_V^L)\lon\RB_\la(V;\rho_V^L,-\rho_V^L)$$ and $$\beta_V:\RB_\la(V;\rho_V^L,-\rho_V^L)\lon(\RB_{\la_\Lie}\circ\huaP_s)(V;\rho_V^L,-\rho_V^L)$$ as following:
\begin{eqnarray}
\label{natural-transformation-1}\alpha_V(g)&=&g,\quad\quad\quad\forall g\in\Hom(\otimes^nV,\Lei(\la)),\\
\label{natural-transformation-2}\beta_V(f)&=&\pr\circ f,\quad\forall f\in\Hom(\otimes^nV,\la).
\end{eqnarray}

\begin{thm}\label{thm:carb}
With the above notations, $\alpha$ is a natural transformation  from the functor $\RB_{\Lei(\la)}\circ\huaI_s$ to $\RB_\la$, and $\beta$ is a natural transformation  from the functor $\RB_\la$ to $\RB_{\la_\Lie}\circ\huaP_s$. Moreover, for any symmetric representation $(V;\rho_V^L,-\rho_V^L)$ of the Leibniz algebra $(\la,[-,-]_\la)$, we have the following short exact sequence of graded Lie algebras:
\begin{eqnarray}\label{ES-Lie}
\qquad0\lon(\RB_{\Lei(\la)}\circ\huaI_s)(V;\rho_V^L,-\rho_V^L)\stackrel{\alpha_V}{\lon}\RB_\la(V;\rho_V^L,-\rho_V^L)\stackrel{\beta_V}{\lon} (\RB_{\la_\Lie}\circ\huaP_s)(V;\rho_V^L,-\rho_V^L)\lon 0.
\end{eqnarray}
\end{thm}

\begin{proof}
Since $\Lei(\la)$ is an ideal of the Leibniz algebra $(\la,[-,-]_\la)$, it follows  that $C^*(V,\Lei(\la))$ is a subalgebra of $\RB_\la(V;\rho_V^L,-\rho_V^L)$. Thus the linear embedding map $\alpha_V$ is a graded Lie algebra homomorphism. Let $\phi:W\lon V$ be a homomorphism of symmetric representations of $(\la,[-,-]_\la)$.  It is straightforward to obtain  the following commutative diagram:
\begin{eqnarray*}
\xymatrix{
 (\RB_{\Lei(\la)}\circ\huaI_s)(V;\rho_V^L,-\rho_V^L) \ar[d]_{ (\RB_{\Lei(\la)}\circ\huaI_s)(\phi)}\ar[rr]^{\,\,\,\,\,\,\,\,\,\,\,\,\,\,\,\,\alpha_V}
                && \RB_\la(V;\rho_V^L,-\rho_V^L)  \ar[d]^{\RB_\la(\phi)}  \\
(\RB_{\Lei(\la)}\circ\huaI_s)(W;\rho_W^L,-\rho_W^L) \ar[rr]^{\,\,\,\,\,\,\,\,\,\,\,\,\,\,\,\,\,\,\alpha_W}
                && \RB_\la(W;\rho_W^L,-\rho_W^L),}
\end{eqnarray*}
which implies that $\alpha$ is a natural transformation.

Since $\pr$ is a Leibniz algebra homomorphism and the definition \eqref{induce-rep} of $\gr_V$,
we deduce that $\beta_V$ is a homomorphism of graded Lie algebras. Moreover, let $\phi:W\lon V$ be a homomorphism of symmetric representations of $(\la,[-,-]_\la)$ and $f\in\Hom(\otimes^nV,\la)$. Then we have
$$
(\pr\circ f)\circ\phi^{\otimes n}=\pr\circ(f\circ \phi^{\otimes n}),
$$
which implies that we have the following commutative diagram:
\begin{eqnarray*}
\xymatrix{
 \RB_\la(V;\rho_V^L,-\rho_V^L) \ar[d]_{ \RB_\la(\phi)}\ar[rr]^{\beta_V}
                && (\RB_{\la_\Lie}\circ\huaP_s)(V;\rho_V^L,-\rho_V^L)  \ar[d]^{(\RB_{\la_\Lie}\circ\huaP_s)(\phi)}  \\
\RB_\la(W;\rho_W^L,-\rho_W^L) \ar[rr]^{\beta_W}
                && (\RB_{\la_\Lie}\circ\huaP_s)(W;\rho_W^L,-\rho_W^L).}
\end{eqnarray*}
Thus,   $\beta$ is a natural transformation.

 Since $V$ is vector space over the field $\K$, for any positive integer $n$, the functor $\Hom(\otimes^nV,-)$ is an exact functor from the category of vector spaces over  $\K$ to itself. Moreover, we have $\alpha_V=\frki_*$ and $\beta_V=\pr_*$. By the exact sequence of Leibniz algebras \eqref{characteristic-element}, we obtain the short exact sequence of graded Lie algebras \eqref{ES-Lie}. The proof is finished.
\end{proof}

\begin{cor}\label{induce-gla-rota-baxter}
  Let $(V;\rho_V^L,-\rho_V^L)$ be  a symmetric representation of a Leibniz algebra $(\la,[-,-]_\la)$. The Maurer-Cartan elements of the graded Lie algebra $(\RB_{\la_\Lie}\circ\huaP_s)(V;\rho_V^L,-\rho_V^L)$ are exactly relative Rota-Baxter operators on the Lie algebra $\la_\Lie$ with respect to the representation $(V,\theta_V)$ given in \eqref{induce-rep}.
\end{cor}

Let $(V;\rho)$ be a representation of a Lie algebra $(\g,[-,-]_\g)$. Then $(V;\rho,-\rho)$ is a symmetric representation of the Leibniz algebra $(\g,[-,-]_\g)$. By  $\Lei(\g)=0$, we obtain a graded Lie algebra $\RB_{\g}(V;\rho,-\rho)$ whose Maurer-Cartan elements are the relative Rota-Baxter operators on the Lie algebra $(\g,[-,-]_\g)$ with respect to the representation $(V;\rho)$. On the other hand, in \cite{TBGS-1} the authors construct another graded Lie algebra on the graded vector subspace $\huaC^*(V,\g):=\oplus_{k\geq 1}\Hom(\wedge^{k}V,\g),$
where the graded Lie bracket is given by
\begin{eqnarray}
&&\nonumber\Courant{P,Q}(v_1,v_2,\cdots,v_{m+n})\\
\label{o-bracket}&=&-\sum_{\sigma\in \mathbb S_{(n,1,m-1)}}(-1)^{\sigma}P(\rho(Q(v_{\sigma(1)},\cdots,v_{\sigma(n)}))v_{\sigma(n+1)},v_{\sigma(n+2)},\cdots,v_{\sigma(m+n)})\\
\nonumber&&+(-1)^{mn}\sum_{\sigma\in \mathbb S_{(m,1,n-1)}}(-1)^{\sigma}Q(\rho(P(v_{\sigma(1)},\cdots,v_{\sigma(m)}))v_{\sigma(m+1)},v_{\sigma(m+2)},\cdots,v_{\sigma(m+n)})\\
\nonumber&&-(-1)^{mn}\sum_{\sigma\in \mathbb S_{(m,n)}}(-1)^{\sigma}[P(v_{\sigma(1)},\cdots,v_{\sigma(m)}),Q(v_{\sigma(m+1)},\cdots,v_{\sigma(m+n)})]_\g
\end{eqnarray}
for all $P\in\Hom(\wedge^mV,\g)$ and $Q\in\Hom(\wedge^nV,\g)$. Its Maurer-Cartan elements are also  relative Rota-Baxter operators on $\g$ with respect to the representation $(V;\rho)$.

\begin{pro}
The graded Lie algebra $(\huaC^*(V,\g),\Courant{\cdot,\cdot})$ is a subalgebra of $\RB_{\g}(V;\rho,-\rho)$.
\end{pro}

\begin{proof}
Let $(V;\rho)$ be a representation of a Lie algebra $(\g,[-,-]_\g)$. The semidirect product Lie algebra $\g\ltimes_{\rho}V$ is the same as the semidirect product Leibniz algebra $\g\ltimes_{\rho,-\rho}V$. We denote this semidirect product multiplication by $\mu.$ Recall that the Nijenhuis-Richardson bracket $[\cdot,\cdot]_{\NR}$ associated to the direct sum vector space $\g\oplus V$ gives rise to a graded Lie algebra $(\oplus_{k\geq 1}\Hom(\wedge^k(\g\oplus V),\g\oplus V),[\cdot,\cdot]_{\NR})$. The graded Lie algebra $(\huaC^*(V,\g),\Courant{\cdot,\cdot})$ is also obtained via the derived bracket \cite{Kosmann-Schwarzbach}:
$$
  \Courant{P,Q}=(-1)^{m-1}[[\mu,P]_{\NR},Q]_{\NR},\quad\forall P\in\Hom(\wedge^mV,\g), Q\in\Hom(\wedge^nV,\g).
$$
Since $(\oplus_{k\geq 1}\Hom(\wedge^k(\g\oplus V),\g\oplus V),[\cdot,\cdot]_{\NR})$ is a graded Lie subalgebra of $(\oplus_{n\ge 1}\Hom(\otimes^n(\g\oplus V),\g\oplus V),[\cdot,\cdot]_{\B})$, we deduce that
\begin{eqnarray*}
\Courant{P,Q}&=&(-1)^{m-1}[[\mu,P]_{\NR},Q]_{\NR}=(-1)^{m-1}[[\mu,P]_{\B},Q]_{\B} =\{P,Q\}_V,
\end{eqnarray*}
which implies  that $(\huaC^*(V,\g),\Courant{\cdot,\cdot})$ is a subalgebra of $(C^*(V,\g),\{-,-\}_V)=\RB_{\g}(V;\rho,-\rho)$.
\end{proof}

\subsection{The controlling algebras of relative averaging operators on a Lie algebra}

Denote by   $\la$-$\ARep$ the category  of   antisymmetric representations of the Leibniz algebra $(\la,[-,-]_\la)$, which is a subcategory of $\la$-$\rep$.   Then we have the following functors
\begin{eqnarray*}
&&\huaI_a:\la\mbox{-}\ARep\lon  \Lei(\la)\mbox{-}\ARep\\
&&\huaP_a:\la\mbox{-}\ARep\lon  \la_\Lie\mbox{-}\ARep
\end{eqnarray*}
which are given by
\begin{itemize}
  \item the functor $\huaI_a:$ $\la\mbox{-}\ARep$ $\to$ $\Lei(\la)\mbox{-}\ARep$, which is defined  on  objects and on  morphisms respectively by
\begin{eqnarray}
\huaI_a\Big((V;\rho_V^L,0)\Big)&=&(V;\rho_V^L\circ\frki,0),\\
\huaI_a(W\stackrel{\phi}{\lon}V)&=&W\stackrel{\phi}{\lon}V,
\end{eqnarray}
 \item the functor $\huaP_a:$ $\la\mbox{-}\ARep$ $\to$ $\la_\Lie\mbox{-}\ARep$, which is defined  on  objects and on morphisms respectively by
\begin{eqnarray}
\huaP_a\Big((V;\rho_V^L,0)\Big)&=&(V;\gr_V,0),\\
\huaP_a(W\stackrel{\phi}{\lon}V)&=&W\stackrel{\phi}{\lon}V,
\end{eqnarray}
\end{itemize}
 for antisymmetric representations  $(W;\rho_W^L,0)$ and $(V;\rho_V^L,0)$   of the Leibniz algebra $(\la,[-,-]_\la)$  and $\phi\in\Hom_{\la\mbox{-}\ARep}(W,V)$.

Thus, we have three functors $\RB_{\Lei(\la)}\circ\huaI_a,~\RB_\la$ and $\RB_{\la_\Lie}\circ\huaP_a$ from the category $\la\mbox{-}\ARep$ to the category $\GLA$. Moreover, for any antisymmetric representation $(V;\rho_V^L,0)$ of the Leibniz algebra $(\la,[-,-]_\la)$, we define $$\alpha_V:(\RB_{\Lei(\la)}\circ\huaI_a)(V;\rho_V^L,0)\lon\RB_\la(V;\rho_V^L,0)$$ and $$\beta_V:\RB_\la(V;\rho_V^L,0)\lon(\RB_{\la_\Lie}\circ\huaP_a)(V;\rho_V^L,0)$$ as following:
\begin{eqnarray*}
\label{natural-transformation-1-anti}\alpha_V(g)&=&g,\quad\quad\quad\forall g\in\Hom(\otimes^nV,\Lei(\la)),\\
\label{natural-transformation-2-anti}\beta_V(f)&=&\pr\circ f,\quad\forall f\in\Hom(\otimes^nV,\la).
\end{eqnarray*}

Similar to Theorem \ref{thm:carb}, we have the following result.

\begin{thm}
With  above  notations, $\alpha$ is a natural transformation  from the functor $\RB_{\Lei(\la)}\circ\huaI_a$ to $\RB_\la$, and $\beta$ is a natural transformation  from the functor $\RB_\la$ to $\RB_{\la_\Lie}\circ\huaP_a$. Moreover, for any antisymmetric representation $(V;\rho_V^L,0)$ of the Leibniz algebra $(\la,[-,-]_\la)$, we have the following short exact sequence of graded Lie algebras:
$$
   0\lon(\RB_{\Lei(\la)}\circ\huaI_a)(V;\rho_V^L,0)\stackrel{\alpha_V}{\lon}\RB_\la(V;\rho_V^L,0)\stackrel{\beta_V}{\lon} (\RB_{\la_\Lie}\circ\huaP_a)(V;\rho_V^L,0)\lon 0.
$$
\end{thm}
\begin{proof}
  The proof is parallel to that of Theorem \ref{thm:carb}, so  we omit the details.
\end{proof}

\begin{cor}\label{induce-gla-anti}
  Let $(V;\rho_V^L,0)$ be  an antisymmetric representation of a Leibniz algebra $(\la,[-,-]_\la)$. The Maurer-Cartan elements of the graded Lie algebra $(\RB_{\la_\Lie}\circ\huaP_a)(V;\rho_V^L,0)$ are exactly relative averaging operators on the Lie algebra $\la_\Lie$ with respect to the representation $(V,\theta_V)$ given in \eqref{induce-rep}.
\end{cor}

Let $(V;\rho)$ be a representation of a Lie algebra $(\g,[-,-]_\g)$. Then $(V;\rho,0)$ is an antisymmetric representation of the Leibniz algebra $(\g,[-,-]_\g)$. By $\Lei(\g)=0$, we obtain a graded Lie algebra $\RB_{\g}(V;\rho,0)$ whose Maurer-Cartan elements are the relative averaging operators on the Lie algebra $(\g,[-,-]_\g)$ with respect to the representation $(V;\rho)$. This graded Lie algebra is exactly the same as the one given in \cite{STZ}.

\section{Relations between the cohomologies}\label{sec:cohomology}

In this section, we establish the relations between the cohomology groups of a relative Rota-Baxter operator  on a Leibniz algebra with respect to a symmetric (resp. antisymmetric) representation and the cohomology groups of the induced  relative Rota-Baxter (resp. averaging) operator  on  the canonical Lie algebra.
\subsection{Cohomology of relative Rota-Baxter operators on Leibniz algebras}

\begin{defi}{\rm (\cite{Loday and Pirashvili})}  Let $(V;\rho^L,\rho^R)$ be a representation of a Leibniz algebra $(\la,[-,-]_{\la})$.
The {\bf  Loday-Pirashvili cohomology} of $\la$ with   coefficients in $V$ is the cohomology of the cochain complex $(C^\bullet(\la,V)=\oplus_{k=0}^{+\infty}C^k(\la,V),\partial)$, where $C^k(\la,V)=
\Hom(\otimes^k\la,V)$ and the coboundary operator
$\partial:C^k(\la,V)\longrightarrow C^
{k+1}(\la,V)$
is defined by
\begin{eqnarray*}
(\partial f)(x_1,\cdots,x_{k+1})&=&\sum_{i=1}^{k}(-1)^{i+1}\rho^L(x_i)f(x_1,\cdots,\hat{x_i},\cdots,x_{k+1})+(-1)^{k+1}\rho^R(x_{k+1})f(x_1,\cdots,x_{k})\\
                      \nonumber&&+\sum_{1\le i<j\le k+1}(-1)^if(x_1,\cdots,\hat{x_i},\cdots,x_{j-1},[x_i,x_j]_\la,x_{j+1},\cdots,x_{k+1}),
\end{eqnarray*}
for all $x_1,\cdots, x_{k+1}\in\la$. The resulting cohomology is denoted by $\huaHL^*(\la,V)$.
\end{defi}

\begin{thm}{\rm(\cite{TSZ})}\label{thm:rep}
Let $T$ be a  relative Rota-Baxter operator on a Leibniz
algebra  $(\la,[-,-]_\la)$ with respect to $(V;\rho^L,\rho^R)$. Define $[-,-]_T:\otimes ^2V\lon V$ by
\begin{eqnarray}\label{thm:rota-baxter-to-leibniz}
[u,v]_{T}=\rho^L(Tu)v+\rho^R(Tv)u,\quad u,v\in V.
\end{eqnarray}
Then $(V,[-,-]_{T})$ is a Leibniz algebra. Moreover,  define $\varrho^L,\varrho^R:V\lon\gl(\la)$ by
\begin{eqnarray}\label{rep-Leibniz}
\varrho^L(u)x=[Tu,x]_\la-T\rho^R(x)u,\quad \varrho^R(u)x=[x,Tu]_\la-T\rho^L(x)u,\quad\forall u\in V,x\in\la.
\end{eqnarray}
Then $(\la;\varrho^L,\varrho^R)$ is a representation of the Leibniz algebra $(V,[-,-]_{T})$.
\end{thm}

Let $\partial_T: C^{n}(V,\la)\longrightarrow  C^{n+1}(V,\la)$ be the corresponding Loday-Pirashvili coboundary operator of the Leibniz algebra $(V,[-,-]_{T})$ with  coefficients in the representation $(\la;\varrho^L,\varrho^R)$. More precisely, $\partial_T: C^{n}(V,\la)\longrightarrow  C^{n+1}(V,\la)$ is given by
               \begin{eqnarray*}
                && (\partial_T f)(v_1,\cdots,v_{n+1})\\&=&\sum_{i=1}^{n}(-1)^{i+1}[Tv_i,f(v_1,\cdots,\hat{v}_i,\cdots, v_{n+1})]_\la-\sum_{i=1}^{n}(-1)^{i+1}T\rho^R(f(v_1,\cdots,\hat{v}_i,\cdots, v_{n+1}))v_i\\
                &&+(-1)^{n+1}[f(v_1,\cdots,v_{n}),Tv_{n+1}]_\la+(-1)^nT\rho^L(f(v_1,\cdots,v_{n}))v_{n+1}\\
                &&+\sum_{1\le i<j\le n+1}(-1)^{i}f(v_1,\cdots,\hat{v}_i,\cdots,v_{j-1},\rho^L(Tv_i)v_j+\rho^R(Tv_j)v_i,v_{j+1},\cdots, v_{n+1}).
               \end{eqnarray*}

\begin{defi}{\rm(\cite{TSZ})}
Let $T$ be a relative Rota-Baxter operator on a Leibniz algebra $(\la,[-,-]_\la)$ with respect to a representation $(V;\rho^L,\rho^R)$. The cohomology of the cochain complex $$(C^*(V,\la)=\oplus _{k=0}^{+\infty}C^k(V,\la),\partial_T)$$ is taken to be  the {\bf cohomology   for the relative Rota-Baxter operator $T$}.
\end{defi}
We denote the set of $k$-cocycles by $\huaZ_{RB}^k(T)$,  the set of $k$-coboundaries by $\huaB_{RB}^k(T)$ and the $k$-th cohomology group by
\begin{eqnarray}\label{de:opcohl2}
\huaH_{RB}^k(T)= \huaZ_{RB}^k(T)/ \huaB_{RB}^k(T).
\end{eqnarray}

Up to a sign, the coboundary operators $\partial_T$ coincides with the differential operator $\{T,-\}_V$ defined by using  the Maurer-Cartan element $T$.

\begin{thm}\label{partial-to-derivation}{\rm(\cite{TSZ})}
 Let $T$ be a relative Rota-Baxter operator on the Leibniz algebra $(\la,[\cdot,\cdot]_\la)$ with respect to the representation $(V;\rho^L,\rho^R)$. Then we have
 $$
{\partial}_T f=(-1)^{n-1}\{T,f\}_V,\quad \forall f\in \Hom(\otimes^nV,\la),~n=1,2,\cdots.
 $$
\end{thm}
\subsection{The Loday-Pirashvili cohomology of relative Rota-Baxter operators on Lie algebras}

In \cite{TBGS-1}  the authors established  a cohomology theory of a relative Rota-Baxter operator  on a Lie algebra as the Chevalley-Eilenberg cohomology of a Lie algebra. Thus we will refer to it as the Chevalley-Eilenberg cohomology of a  relative Rota-Baxter operator on a Lie algebra.  More precisely, let $T$ be a relative Rota-Baxter operator on a Lie
algebra $\g$ with respect to a representation $(V;\rho)$.
  Define $\varrho:V\longrightarrow\gl(\g)$ by
 \begin{equation*}
   \varrho(u)(x):=[Tu,x]+T\rho(x)(u),\;\;\forall x\in \g,u\in
    V.
  \end{equation*}
Then  $(\g;\varrho)$ is a representation of the Lie algebra $(V,[-,-]_T)$ on the vector space $\g$. We denote by $(\frkC^*(V,\g),d_\CE)$ the Chevalley-Eilenberg cochain complex of the Lie algebra $(V,[-,-]_{T})$ with  coefficients in the representation $(\g;\varrho)$, where $\frkC^*(V,\g)=\oplus _{k\geq 0}\frkC^k(V,\g)$ and $\frkC^k(V,\g)=\Hom(\wedge^kV,\g)$. The cohomology of the cochain complex $(\frkC^*(V,\g),d_\CE)$ is called  the {\bf  Chevalley-Eilenberg cohomology of the relative Rota-Baxter operator } $T$.  Denote the set of $k$-cocycles by $\huaZ^k_\CE(T)$ and the set of $k$-coboundaries by $\huaB^k_\CE(T)$. Denote by
  \begin{equation}
  \huaH^k_\CE(T)=\huaZ^k_\CE(T)/\huaB^k_\CE(T), \quad k \geq 0,
  \label{eq:ocoh}
  \end{equation}
the corresponding $k$-th cohomology group.

As explained in Section \ref{sec:L}, a relative Rota-Baxter operator $T:V\lon \g$ on a Lie algebra $(\g,[-,-]_\g)$ with respect to a representation  $(V;\rho)$ can be viewed as a special relative Rota-Baxter operator  on  the Leibniz algebra $(\g,[-,-]_\g)$ with respect to the symmetric representation   $(V;\rho,-\rho)$. Thus, the above approach to define the cohomology of a relative Rota-Baxter operator  on a Leibniz algebra  can be applied to define a cohomology of a relative  Rota-Baxter operator  on a Lie algebra. More precisely, let $T:V\lon \g$ be a relative  Rota-Baxter operator on a Lie algebra $(\g,[-,-]_\g)$ with respect to a representation  $(V;\rho)$. Then $(V,[-,-]_T)$ is a Leibniz algebra, where the Leibniz bracket $[-,-]_T$ is defined by
$$
[u,v]_T=\rho(Tu)v-\rho(Tv)u,\quad \forall u,v\in V.
$$
Moreover, $(\g;\varrho^L,\varrho^R)$ is a symmetric representation of the Leibniz algebra $(V,[-,-]_T)$, where $\varrho^L,\varrho^R: V\lon\gl(\g)$ is defined by
$$
\varrho^L(u)x=[Tu,x]_\g+T\rho(x)u,\quad \varrho^R(u)x=[x,Tu]_\g-T\rho(x)u.
$$

Let $\partial_T: C^{n}(V,\g)\longrightarrow  C^{n+1}(V,\g)$ be the corresponding Loday-Pirashvili coboundary operator of the Leibniz algebra $(V,[-,-]_{T})$ with  coefficients in the symmetric representation $(\g;\varrho^L,\varrho^R)$.

 \begin{defi}
    Let $T:V\lon \g$ be a relative Rota-Baxter operator on a Lie algebra $(\g,[-,-]_\g)$ with respect to a representation  $(V;\rho)$.  The cohomology of the cochain complex $(\oplus_{n=0}^{+\infty} C^{n}(V,\g),\partial_T)$ is called the {\bf Loday-Pirashvili cohomology of the relative  Rota-Baxter operator} $T$.
 \end{defi}

 We denote the set of $k$-cocycles by $\huaZ^k_{\LP}( {T})$,  the set of $k$-coboundaries by $\huaB^k_{\LP}( {T})$ and the $k$-th cohomology group by
\begin{eqnarray}\label{de:opcohaLP}
\huaH^k_{\LP}( {T})=\huaZ^k_{\LP}( {T})/ \huaB^k_{\LP}( {T}).
\end{eqnarray}

\begin{rmk}
Note that the monomorphism $\frkC^k(V,\g)=\Hom(\wedge^kV,\g)\to C^k(V,\g)=\Hom(\otimes^kV,\g)$ induced from the natural epimorphism $\otimes^k V\to\wedge^kV$ induces isomorphisms
$$\huaH^0_{\CE}( {T})\cong\huaH^0_{\LP}( {T})\,\,\,{\rm and}\,\,\,\huaH^1_{\CE}( {T})\cong\huaH^1_{\LP}( {T})$$
and a long exact sequence
\begin{eqnarray*}
0 & \to & \huaH^2_{\CE}( {T})\to\huaH^2_{\LP}( {T})\to\huaH_{\rm rel}^0({T})\\
& \to & \huaH^3_{\CE}( {T})\to\huaH^3_{\LP}( {T})\to\huaH_{\rm rel}^1({T}\to\cdots
\end{eqnarray*}
where $\huaH_{\rm rel}^\bullet({T})$ is the cohomology of the {\it relative complex}
$C_{\rm rel}^\bullet(V,\g)$ which is by definition (up to a degree shift) the cokernel of the monomorphism
$\Hom(\wedge^kV,\g)\to\Hom(\otimes^kV,\g)$. These matters as well as a spectral sequence linking $\huaH_{\rm rel}^\bullet({T})$ and $\huaH_{\LP}^\bullet({T})$ in still another way can be found in \cite{FW}.
In particular, Theorem 2.6 in \cite{FW} shows how the vanishing of Chevalley-Eilenberg cohomology implies the vanishing of Loday-Pirashvili cohomology.
 \end{rmk}

 \begin{rmk}
 Note furthermore that the comparison between Chevalley-Eilenberg cohomologies and Loday-Pirashvili cohomologies makes sense more generally for any relative Rota-Baxter operator $T:V\to\la$ over a Leibniz algebra $\la$ as soon as the representation of $\la$ on $V$ is symmetric, because then the bracket $[-,-]_T$ on $V$ is a Lie bracket.

 In case the induced representation of $(V,[-,-]_T)$ on $\la$ is not symmetric, one may use the short exact sequence
 $$0\to \la_{\rm anti}\to \la\to \la_{\rm sym}\to 0$$
 and the relation of Loday-Pirashvili cohomology with values in an antisymmetric representation to the cohomology with values in a symmetric representation (see \cite{FW}) in order to reduce the computations to the case of symmetric coefficients.
 \end{rmk}

 Let $T:V\lon\la$ be a relative Rota-Baxter operator on a Leibniz algebra $\la$ with respect to   a symmetric representation $(V;\rho^L,\rho^R=-\rho^L)$. By Proposition \ref{pro:indRB},  $\bar{T}:=\pr\circ T:V\lon\la_\Lie$ is a relative Rota-Baxter operator on the Lie algebra $\la_\Lie$ with respect to the representation $(V;\gr)$ given by \eqref{induce-rep}. Now we establish the relationship between the cohomology groups of  the relative Rota-Baxter operator $T$ and the Loday-Pirashvili cohomology groups of  the relative Rota-Baxter operator $\bar{T}.$

\begin{lem}\label{lem:subspacesym}
  Let $T:V\lon\la$ be a relative Rota-Baxter operator on a Leibniz algebra $\la$ with respect to a symmetric representation $(V;\rho^L,\rho^R=-\rho^L)$. Then $(\Lei(\la);\varrho^L,\varrho^R)$ is a subrepresentation of $(\la;\varrho^L,\varrho^R)$.
\end{lem}
\begin{proof}
  Since $\rho^R=-\rho^L$, by \eqref{lem:imp}, we have
  \begin{eqnarray*}
    \varrho^L(u)[y,y]_\la=[Tu,[y,y]_\la]_\la+\rho^L([y,y]_\la)u
    =[[Tu,y]_\la,y]_\la+[y,[Tu,y]_\la]_\la,
  \end{eqnarray*}
  for all $y\in\la$, which implies that $\varrho^L(u)(\Lei(\la))\subset \Lei(\la).$

  By the fact that $\Lei(\la)$ is contained in the left center of the Leibniz algebra $\la$, we have
  \begin{eqnarray*}
     \varrho^R(u)[y,y]_\la=[[y,y]_\la,Tu]_\la-T\rho^L([y,y]_\la)u=0,
  \end{eqnarray*}
  which implies that $\varrho^R(u)(\Lei(\la))=0\subset \Lei(\la)$. Therefore, $(\Lei(\la);\varrho^L,\varrho^R)$ is a subrepresentation of $(\la;\varrho^L,\varrho^R)$.
\end{proof}

Note that even though the Leibniz algebra $(V,[\cdot,\cdot]_{T})$ becomes a Lie algebra if $T$ is a relative Rota-Baxter operator on a Leibniz algebra $\la$ with respect to a symmetric representation, the representation given in Theorem \ref{thm:rep} is still a representation of $(V,[\cdot,\cdot]_{T})$ as a Leibniz algebra, and $\varrho^R\neq -\varrho^L.$ But the quotient representation $(\la_\Lie;\grl,\grr)$  is symmetric, where
$\grl:V\lon\gl(\la_\Lie)$ and $\grr:V\lon\gl(\la_\Lie)$ are given by
\begin{equation}\label{eq:repquo}
  \grl(u)\bar{x}=\overline{\varrho^L(u)x},\quad \grr(u)\bar{x}=\overline{\varrho^R(u)x},\quad \forall u\in V, x\in\la.
\end{equation}

\begin{pro}
  Let $T:V\lon\la$ be a relative Rota-Baxter operator on a Leibniz algebra $\la$ with respect to a symmetric representation $(V;\rho^L,\rho^R=-\rho^L)$. Then the quotient representation $(\la_\Lie;\grl,\grr)$ is a symmetric representation of $(V,[\cdot,\cdot]_{T})$ as a Leibniz algebra, i.e. $ \grr=-\grl$.
\end{pro}

\begin{proof}
For all $u\in V,x\in\la$, we have
\begin{eqnarray*}
\varrho^L(u)x+\varrho^R(u)x&=&[Tu,x]_\la-T\rho^R(x)u+[x,Tu]_\la-T\rho^L(x)u\\
                            &=&[Tu,x]_\la+[x,Tu]_\la\in\Lei(\la).
\end{eqnarray*}
Thus, we deduce that $\grr=-\grl$.
\end{proof}

Let $\overline{\partial}_{{T}}: C^{n}(V,\la_\Lie)\longrightarrow  C^{n+1}(V,\la_\Lie)$ be the corresponding Loday-Pirashvili coboundary operator of the Leibniz algebra $(V,[\cdot,\cdot]_{T})$ with  coefficients in the quotient representation $(\la_\Lie;\grl,\grr)$.

\begin{pro}\label{pro:samecomplex1}
Let $T$ be a relative Rota-Baxter operator on a Leibniz algebra $(\la,[\cdot,\cdot]_\la)$ with respect to a symmetric representation $(V;\rho^L,\rho^R=-\rho^L)$. Then
 $$
 \overline{\partial}_{{T}}=\partial_{\bar{T}}.
 $$
 That is, the Loday-Pirashvili cochain complex $(\oplus_{n=0}^{+\infty} C^{n}(V,\la_\Lie),\partial_{\bar{T}})$ associated to the relative Rota-Baxter operator $\bar{T}$ and the cochain complex $(\oplus_{n=0}^{+\infty} C^{n}(V,\la_\Lie),\overline{\partial}_{T})$ obtained by using the quotient representation $(\la_\Lie;\grl,\grr)$ are the same.
\end{pro}

\begin{proof}Obviously, the spaces of $n$-cochains are the same. Thus we only need to show that the induced Leibniz algebra structure $[-,-]_{\bar{T}}$ on $V$ by the relative Rota-Baxter operator $\bar{T}$ is the same as $[-,-]_{{T}}$, and the representations $(\la_\Lie;\varrho^L,\varrho^R)$ and $(\la_\Lie;\grl,\grr)$ are the same.

For all $u,v\in V$, by \eqref{induce-rep}, we have
$$
[u,v]_{\bar{T}}=\theta(\bar{T}u)v-\theta(\bar{T}v)u=\rho^L(Tu)v-\rho^L(Tv)u=[u,v]_T.
$$

For any $u\in V,x\in\la$, we have
\begin{eqnarray*}
&&\varrho^L(u)(\bar{x})=[\bar{T}u,\bar{x}]_{\la_\Lie}+\bar{T}\gr(\bar{x})u=\overline{[Tu,x]_{\la}}+\overline{T\rho^L(x)u}=\grl(u)\bar{x},\\
&&\varrho^R(u)(\bar{x})=[\bar{x},\bar{T}u]_{\la_\Lie}-\bar{T}\gr(\bar{x})u=\overline{[x,Tu]_{\la}}-\overline{T\rho^L(x)u}=\grr(u)\bar{x}.
\end{eqnarray*}
Thus $
 \overline{\partial}_{{T}}=\partial_{\bar{T}}.
 $
\end{proof}

Let $T:V\lon\la$ be a relative Rota-Baxter operator on a Leibniz algebra $\la$ with respect to a symmetric representation $(V;\rho^L,\rho^R=-\rho^L)$. By Lemma \ref{lem:subspacesym}, we have the following short exact sequence of representations of the Leibniz algebra $(V,[\cdot,\cdot]_{T})$:
\begin{eqnarray}\label{ES-Lie-anti-rep-rb}
0\lon(\Lei(\la);\varrho^L,\varrho^R)\stackrel{\frki}{\lon}(\la;\varrho^L,\varrho^R)\stackrel{\pr}{\lon} (\la_\Lie;\grl,\grr)\lon 0.
\end{eqnarray}
Moreover, we have the following result:

\begin{thm}\label{thm:sequence-coh-rb}
Let $T$ be a relative Rota-Baxter operator on a Leibniz algebra $(\la,[\cdot,\cdot]_\la)$ with respect to a symmetric representation $(V;\rho^L,\rho^R=-\rho^L)$. Then there is a short exact sequence of the  cochain complexes:
\[
\small{ \xymatrix{0\ar[d] & 0\ar[d]\\
\cdots
\longrightarrow \Hom(\otimes^nV,\Lei(\la))\ar[d]^{\alpha_V} \ar[r]^{\partial} & \Hom(\otimes^{n+1}V,\Lei(\la)) \ar[d]^{\alpha_V} \longrightarrow\cdots  \\
\cdots
\longrightarrow \Hom(\otimes^nV,\la)\ar[d]^{\beta_V} \ar[r]^{\partial_T} & \Hom(\otimes^{n+1}V,\la) \ar[d]^{\beta_V} \longrightarrow\cdots\\
\cdots
\longrightarrow \Hom(\otimes^nV,\la_\Lie)\ar[d] \ar[r]^{\partial_{\bar{T}}} & \Hom(\otimes^{n+1}V,\la_\Lie)\ar[d]  \longrightarrow\cdots,\\
0 & 0
}}
\]
where $\alpha_V$ and $\beta_V$ are given by \eqref{natural-transformation-1} and \eqref{natural-transformation-2} respectively.

Consequently, there is a long exact sequence of the  cohomology groups:
\begin{equation}\label{long-exact-of-cohomology-rb}
\cdots\longrightarrow \huaHL^n(V,\Lei(\la))\stackrel{\huaH^n(\alpha_V)}{\longrightarrow}\huaH_{RB}^n(T)\stackrel{\huaH^n(\beta_V)}{\longrightarrow} \huaH_{\LP}^n(\bar{T})\stackrel{c^n}\longrightarrow \huaHL^{n+1}(V,\Lei(\la))\longrightarrow\cdots,
\end{equation}
where the connecting map $c^n$ is defined by
$c^n([h])=[\alpha_V^{-1}(\partial_T(\beta_V^{-1}(h)))],$  for all $[h]\in \huaH_{\LP}^n(\bar{T}).$
\end{thm}

\begin{proof}
For any $g\in \Hom(\otimes^nV,\Lei(\la))$ and $v_1,\cdots,v_{n+1}\in V$, we have
\begin{eqnarray*}
                && \big(\alpha_V (\partial g)\big)(v_1,\cdots,v_{n+1})\\
                &\stackrel{\eqref{natural-transformation-1}}{=}&(\partial g)(v_1,\cdots,v_{n+1})\\
                &=&\sum_{i=1}^{n}(-1)^{i+1}[Tv_i,g(v_1,\cdots,\hat{v}_i,\cdots, v_{n+1})]_\la+(-1)^nT\rho^L(g(v_1,\cdots,v_{n}))v_{n+1}\\
                &&+\sum_{1\le i<j\le n+1}(-1)^{i}g(v_1,\cdots,\hat{v}_i,\cdots,v_{j-1},\rho^L(Tv_i)v_j+\rho^R(Tv_j)v_i,v_{j+1},\cdots, v_{n+1})\\
                &=&\big(\partial_T(\alpha_V(g))\big)(v_1,\cdots,v_{n+1}),
               \end{eqnarray*}
which implies that $\alpha_V$ is a homomorphism of cochain complexes. For  $n=0$ and $x\in\la,v\in V$, we have
\begin{eqnarray*}
\big(\beta_V(\partial_T x)\big)(v)=\beta_V(-[x,Tv]_\la+T\rho^L(x)v)
\stackrel{\eqref{natural-transformation-2}}{=}-\grr(v)\bar{x}=\big(\partial_{\bar{T}} \beta_V (x)\big)(v).
\end{eqnarray*}
For $n\ge 1$, by Theorem \ref{partial-to-derivation} and Theorem \ref{thm:carb}, for all  $f\in \Hom(\otimes^nV,\la)$, we have
\begin{eqnarray*}
 \beta_V(\partial_T f)=(-1)^{n-1}\beta_V\{T,f\}_V=(-1)^{n-1}\{\beta_V(T),\beta_V(f)\}_V=\partial_{\bar{T}} \beta_V (f).
\end{eqnarray*}
Thus, we deduce that $\beta_V$ is a homomorphism of cochain complexes. Since $V$ is vector space over the field $\K$, for any positive integer $n$, the functor $\Hom(\otimes^nV,-)$ is an exact functor from the category of vector spaces over  $\K$ to itself. Moreover, we have $\alpha_V=\frki_*$ and $\beta_V=\pr_*$.  By the  short exact sequence \eqref{ES-Lie-anti-rep-rb} of representations of the Leibniz algebra $(V,[-,-]_{T})$, we obtain the short exact sequence of the above cochain complexes. The proof is finished.
\end{proof}

\begin{rmk}
In the situation of Theorem \ref{thm:sequence-coh-rb}, the short exact sequence of cochain complexes can be completed into a diagram
\[
\small{ \xymatrix{
            &    0 \ar[d] & 0 \ar[d] & 0 \ar[d] & \\
0 \ar[r] & \frkC^\bullet(V,\Lei(\la)) \ar[d]  \ar[r] & C^\bullet(V,\Lei(\la)) \ar[d] \ar[r] & C_{\rm rel}^\bullet(V,\Lei(\la)) \ar[d] \ar[r] & 0 \\
0 \ar[r] &  \frkC^\bullet(V,\la) \ar[r]  \ar[d]   & C^\bullet(V,\la) \ar[d] \ar[r] & C_{\rm rel}^\bullet(V,\la) \ar[d] \ar[r] & 0 \\
0 \ar[r] &  \frkC^\bullet(V,\la_\Lie)  \ar[r]  \ar[d] & C^\bullet(V,\la_\Lie) \ar[d]  \ar[r] & C_{\rm rel}^\bullet(V,\la_\Lie) \ar[r] \ar[d]  & 0 \\
   &    0  & 0 & 0  & \\
}}
\]
which leads then to a corresponding exact diagram in cohomology (starting from degree $2$, while there are isomorphisms in degree $0$ and $1$ as before).

\end{rmk}

\subsection{Cohomology of relative averaging operators on Lie algebras}

As explained in Section \ref{sec:L}, a relative  averaging operator $T:V\lon \g$ on a Lie algebra $(\g,[-,-]_\g)$ with respect to a representation  $(V;\rho)$ can be viewed as a special relative Rota-Baxter operator  on  the Leibniz algebra $(\g,[-,-]_\g)$ with respect to the antisymmetric representation   $(V;\rho,0)$. Thus, the above approach to define the cohomology of relative Rota-Baxter operators on Leibniz algebras can be applied to define the cohomology of relative  averaging operators. More precisely, let $T:V\lon \g$ be a relative  averaging operator on a Lie algebra $(\g,[-,-]_\g)$ with respect to a representation  $(V;\rho)$. Then $(V,[-,-]_T)$ is a Leibniz algebra, where the Leibniz bracket $[-,-]_T$ is defined by
$$
[u,v]_T=\rho(Tu)v,\quad \forall u,v\in V.
$$
Moreover, $(\g;\varrho^L,\varrho^R)$ is a representation of the Leibniz algebra $(V,[-,-]_T)$, where $\varrho^L,\varrho^R: V\lon\gl(\g)$ is defined by
$$
\varrho^L(u)x=[Tu,x]_\g,\quad \varrho^R(u)x=[x,Tu]_\g-T\rho(x)u.
$$

Let $\partial_T: C^{n}(V,\g)\longrightarrow  C^{n+1}(V,\g)$ be the corresponding Loday-Pirashvili coboundary operator of the Leibniz algebra $(V,[-,-]_{T})$ with  coefficients in the representation $(\g;\varrho^L,\varrho^R)$. More precisely, $\partial_T: C^{n}(V,\g)\longrightarrow  C^{n+1}(V,\g)$ is given by
               \begin{eqnarray*}
               (\partial_T f)(v_1,\cdots,v_{n+1}) &=&\sum_{i=1}^{n}(-1)^{i+1}[Tv_i,f(v_1,\cdots,\hat{v}_i,\cdots, v_{n+1})]_\g \\
               && +(-1)^{n+1}[f(v_1,\cdots,v_{n}),Tv_{n+1}]_\g+(-1)^nT\rho(f(v_1,\cdots,v_{n}))v_{n+1}\\
                &&+\sum_{1\le i<j\le n+1}(-1)^{i}f(v_1,\cdots,\hat{v}_i,\cdots,v_{j-1},\rho (Tv_i)v_j,v_{j+1},\cdots, v_{n+1}).
               \end{eqnarray*}

 \begin{defi}
    Let $T:V\lon \g$ be a relative  averaging operator on a Lie algebra $(\g,[-,-]_\g)$ with respect to a representation  $(V;\rho)$.  The cohomology of the cochain complex $(\oplus_{n=0}^{+\infty} C^{n}(V,\g),\partial_T)$ is defined to be the {\bf cohomology of the relative  averaging operator} $T$.
 \end{defi}

 We denote the set of $k$-cocycles by $\huaZ^k_{AO}( {T})$,  the set of $k$-coboundaries by $\huaB^k_{AO}( {T})$ and the $k$-th cohomology group by
\begin{eqnarray}\label{de:opcoha1}
\huaH^k_{AO}( {T})=\huaZ^k_{AO}( {T})/ \huaB^k_{AO}( {T}).
\end{eqnarray}
 See  \cite{STZ} for more details about the cohomology theory of a relative averaging operator on a Lie algebra.

Let $T:V\lon\la$ be a relative Rota-Baxter operator on the Leibniz algebra $\la$ with respect to   an antisymmetric representation $(V;\rho^L,\rho^R=0)$. By Proposition \ref{pro:indET},  $\bar{T}:=\pr\circ T:V\lon\la_\Lie$ is a relative averaging operator on the Lie algebra $\la_\Lie$ with respect to the representation $(V;\gr)$ given by \eqref{induce-rep}. Now we establish the relationship between the cohomology groups of  the relative Rota-Baxter operator $T$ and the cohomology groups of  the relative averaging operator $\bar{T}.$

Similarly as Lemma \ref{lem:subspacesym}, we have the following result.
\begin{lem}\label{lem:subspaceanti}
  Let $T:V\lon\la$ be a relative Rota-Baxter operator on a Leibniz algebra $\la$ with respect to an antisymmetric representation $(V;\rho^L,\rho^R=0)$. Then $(\Lei(\la);\varrho^L,\varrho^R)$ is a subrepresentation of $(\la;\varrho^L,\varrho^R)$.

\end{lem}

By Lemma \ref{lem:subspaceanti}, we have the quotient representation $(\la_\Lie;\grl,\grr)$, where
$\grl:V\lon\gl(\la_\Lie)$ and $\grr:V\lon\gl(\la_\Lie)$ are given by \eqref{eq:repquo}.
 Let $\overline{\partial}_{{T}}: C^{n}(V,\la_\Lie)\longrightarrow  C^{n+1}(V,\la_\Lie)$ be the corresponding Loday-Pirashvili coboundary operator of the Leibniz algebra $(V,[-,-]_{T})$ with  coefficients in the quotient representation $(\la_\Lie;\grl,\grr)$. Similar to Proposition \ref{pro:samecomplex1}, we have the following result.

\begin{pro}\label{pro:samecomplex}
Let $T$ be a relative Rota-Baxter operator on a Leibniz algebra $(\la,[-,-]_\la)$ with respect to an antisymmetric representation $(V;\rho^L,\rho^R=0)$. Then
 $$
 \overline{\partial}_{{T}}=\partial_{\bar{T}}.
 $$
 That is, the cochain complex $(\oplus_{n=0}^{+\infty} C^{n}(V,\la_\Lie),\partial_{\bar{T}})$ associated to the relative averaging operator $\bar{T}$ and the cochain complex $(\oplus_{n=0}^{+\infty} C^{n}(V,\la_\Lie),\overline{\partial}_{T})$ obtained by using the quotient representation $(\la_\Lie;\grl,\grr)$ are the same.
\end{pro}

\begin{proof}
Obviously, the spaces of $n$-cochains are the same. Thus we only need to show that the induced Leibniz algebra structure $[-,-]_{\bar{T}}$ on $V$ by the relative averaging operator $\bar{T}$ is the same as $[-,-]_{{T}}$, and the representations $(\la_\Lie;\varrho^L,\varrho^R)$ and $(\la_\Lie;\grl,\grr)$ are the same.

For all $u,v\in V$, by \eqref{induce-rep}, we have
$$
[u,v]_{\bar{T}}=\theta(\bar{T}u)v=\rho^L(Tu)v=[u,v]_T.
$$

 For any $u\in V,x\in\la$, we have
\begin{eqnarray*}
&&\varrho^L(u)\bar{x}=[\bar{T}u,\bar{x}]_{\la_\Lie}=\overline{[Tu,x]_{\la}}=\grl(u)\bar{x},\\
&&\varrho^R(u)\bar{x}=[\bar{x},\bar{T}u]_{\la_\Lie}-\bar{T}\gr(\bar{x})u=\overline{[x,Tu]_{\la}}-\overline{T\rho^L(x)u}=\grr(u)\bar{x}.
\end{eqnarray*}
Thus, $\overline{\partial}_{{T}}=\partial_{\bar{T}}.$
\end{proof}

Let $T:V\lon\la$ be a relative Rota-Baxter operator on a Leibniz algebra $\la$ with respect to an antisymmetric representation $(V;\rho^L,\rho^R=0)$. By Lemma \ref{lem:subspaceanti}, we have the following short exact sequence of representations of the Leibniz algebra $(V,[-,-]_{T})$:
\begin{eqnarray}\label{ES-Lie-anti-rep-av}
0\lon(\Lei(\la);\varrho^L,\varrho^R)\stackrel{\frki}{\lon}(\la;\varrho^L,\varrho^R)\stackrel{\pr}{\lon} (\la_\Lie;\grl,\grr)\lon 0.
\end{eqnarray}


Similar to Theorem \ref{thm:sequence-coh-rb}, we have the following result describing the relation among various cohomology groups.

\begin{thm}Let $T$ be a relative Rota-Baxter operator on a Leibniz algebra $(\la,[\cdot,\cdot]_\la)$ with respect to an antisymmetric representation $(V;\rho^L,\rho^R=0)$. Then we have a long exact sequence of cohomology groups:
\begin{equation}\label{long-exact-of-cohomology-avv}
\cdots\longrightarrow \huaHL^n(V,\Lei(\la))\stackrel{\huaH^n(\alpha_V)}{\longrightarrow}\huaH_{RB}^n(T)\stackrel{\huaH^n(\beta_V)}{\longrightarrow} \huaH_{AO}^n(\bar{T})\stackrel{c^n}\longrightarrow \huaHL^{n+1}(V,\Lei(\la))\longrightarrow\cdots.
\end{equation}
\end{thm}

\emptycomment{
\begin{thm}Let $T$ be a relative Rota-Baxter operator on a Leibniz algebra $(\la,[\cdot,\cdot]_\la)$ with respect to an antisymmetric representation $(V;\rho^L,\rho^R=0)$. Then the map $\pr$ is cochain map from the cochain complex $(C^*(V,\la )=\oplus _{k=0}^{+\infty}C^k(V,\la ),\partial_{ {T}})$ to the cochain complex $(C^*(V,\la_\Lie)=\oplus _{k=0}^{+\infty}C^k(V,\la_\Lie),\partial_{\bar{T}})$. That is, we have the following commutative diagram:
$$
\xymatrix{
C^k(V,\la)\ar[d]_{\partial_T}  \ar[r]^{\pr } & C^k(V,\la_\Lie) \ar[d]^{\partial_{\bar{T}}}  \\
C^{k+1}(V,\la )\ar[r]^{\pr }   &C^{k+1}(V,\la_\Lie).
}$$

Consequently, the cochain map $\pr$ induces a homomorphism $\pr_*$ from the cohomology group $\huaH^k_{{T}}(V,\la)$ to $\huaH^k_{\bar{T}}(V,\la_\Lie)$
\end{thm}
}

\subsection{Some computations of cohomologies}

\begin{ex}\label{example-RB-1}{\rm
Let $\gl(n,\mathbb C)$ be the Lie algebra of all $n\times n$ matrices over $\mathbb C$. Then $T=\Id:\gl(n,\mathbb C)\lon\gl(n,\mathbb C)$ is a relative Rota-Baxter operator on the Lie algebra $\gl(n,\mathbb C)$ with respect to the representation $(\gl(n,\mathbb C);\rho)$, where $\rho:\gl(n,\mathbb C)\lon\gl(\gl(n,\mathbb C))$ is defined by
  $$
  \rho(A)B=AB,\quad \forall A,B\in \gl(n,\mathbb C).
  $$
Then $(\gl(n,\mathbb C),[-,-]_T)$ is the original Lie algebra $\gl(n,\mathbb C)$. Moreover, $(\gl(n,\mathbb C);\varrho)$ is a representation of the Lie algebra $\gl(n,\mathbb C)$, where $\varrho: \gl(n,\mathbb C)\lon\gl(\gl(n,\mathbb C))$ is defined by
\begin{equation}
\varrho(A)B:=[TA,B]+T\rho(B)(A)=AB=\rho(A)B,\;\;\forall A,B\in \gl(n,\mathbb C).
  \end{equation}
Thus, the cohomology of the relative Rota-Baxter operator $T$ is the Chevalley-Eilenberg cohomology of $\gl(n,\mathbb C)$ with coefficients in the representation $(\gl(n,\mathbb C);\rho)$. In fact, the Lie algebra $\gl(n,\mathbb C)$ is the direct sum $\gl(n,\mathbb C)=\sln(n,\mathbb C)\oplus{\mathbb C}$, where the factor ${\mathbb C}$ is the center and consists of diagonal matrices. The Hochschild-Serre formula \cite{HS} reads then
\begin{align*}
\huaH_\CE^n(\gl(n,\mathbb C),\gl(n,\mathbb C))=\bigoplus_{p+q=n}\huaH_\CE^p(\sln(n,\mathbb C))\otimes \huaH_\CE^q({\mathbb C},\gl(n,\mathbb C))^{\sln(n,\mathbb C)} \\
= \huaH_\CE^n(\sln(n,\mathbb C))\otimes \huaH_\CE^0({\mathbb C},\gl(n,\mathbb C))^{\sln(n,\mathbb C)}\oplus
\huaH_\CE^{n-1}(\sln(n,\mathbb C))\otimes \huaH_\CE^1({\mathbb C},\gl(n,\mathbb C))^{\sln(n,\mathbb C)}.
\end{align*}

The invariant space $\huaH_\CE^0({\mathbb C},\gl(n,\mathbb C))^{\sln(n,\mathbb C)}$ is zero, because $\huaH_\CE^0({\mathbb C},\gl(n,\mathbb C))=0$. Furthermore, we have $\huaH_\CE^1({\mathbb C},\gl(n,\mathbb C))=0$, because all elements are coboundaries. Therefore all the cohomology groups are zero.
}
\end{ex}

\begin{ex}\label{example-RB-2}{\rm
Let $\tn(n,\mathbb C)$ be the Lie algebras of all $n\times n$ upper triangular matrices over $\mathbb C$. Then $T=\Id:\tn(n,\mathbb C)\lon\tn(n,\mathbb C)$ is a relative Rota-Baxter operator on the Lie algebra $\tn(n,\mathbb C)$ with respect to the representation $(\tn(n,\mathbb C);\rho)$, where $\rho:\tn(n,\mathbb C)\lon\gl(\tn(n,\mathbb C))$ is defined by
  $$
  \rho(A)B=AB,\quad \forall A,B\in \tn(n,\mathbb C).
  $$
Thus, the cohomology of the relative Rota-Baxter operator $T$ is the Chevalley-Eilenberg cohomology of $\tn(n,\mathbb C)$ with coefficients in the representation $(\tn(n,\mathbb C);\rho)$. Recall that the Lie algebra $\tn(n,\mathbb C)$ is an extension
$$0\to \n(n,\mathbb C)\to \tn(n,\mathbb C)\to \dn(n,\mathbb C) \to 0,$$
where $\n(n,\mathbb C)$ is the Lie algebras of all $n\times n$ upper triangular nilpotent matrices and
$\dn(n,\mathbb C)$ is the Lie algebras of all $n\times n$ diagonal matrices.

The Hochschild-Serre spectral sequence \cite{HS} associated to this extension converges towards $\huaH_\CE^\bullet(\tn(n,\mathbb C),\tn(n,\mathbb C))$ and has as its second page
$$E_2^{p,q}=\huaH_\CE^p(\dn(n,\mathbb C), \huaH_\CE^q(\n(n,\mathbb C),\tn(n,\mathbb C))).$$
Let us compute this $E_2$-term for $n=2$. The Lie algebra $\n(2,\mathbb C)$ is $1$-dimensional and we have only non zero contributions from $q=0,1$. We have $\huaH_\CE^0(\n(2,\mathbb C),\tn(2,\mathbb C))=\tn(2,\mathbb C)^{\n(2,\mathbb C)}$ of dimension $2$ generated by
$$
\left(\begin{array}{cc} 1 & 0 \\ 0 & 0 \end{array}\right)\,\,\,\,{\rm and}\,\,\,\,\left(\begin{array}{cc} 0 & 1 \\ 0 & 0 \end{array}\right).
$$

From there, we obtain using Dixmier's Theorem 1 of \cite{Dix} for the nilpotent Lie algebra $\dn(2,\mathbb C)$ with values in a module which does not contains the trivial module
$$\huaH_\CE^p(\dn(2,\mathbb C), \huaH_\CE^0(\n(2,\mathbb C),\tn(2,\mathbb C)))=0,$$
for all $p$. On the other hand, $\huaH_\CE^1(\n(2,\mathbb C),\tn(2,\mathbb C))\cong\dn(2,\mathbb C)$ as all maps in $\Hom(\n(2,\mathbb C),\tn(2,\mathbb C))\cong\tn(2,\mathbb C)$ are cocycles and the coboundaries correspond to $\n(2,\mathbb C)$. From there, we obtain again with Dixmier's Theorem 1
$$\huaH_\CE^p(\dn(2,\mathbb C), \huaH_\CE^1(\n(2,\mathbb C),\tn(2,\mathbb C)))=0,$$
for all $p$. Thus again all the cohomology groups are zero. For general $n$, the whole cohomology will still vanish in case the cohomology spaces $\huaH_\CE^q(\n(n,\mathbb C),\tn(n,\mathbb C))$ do not contain the trivial $\dn(n,\mathbb C)$-module.

}
\end{ex}

\begin{ex}\label{example-RB-3}{\rm
Let $\n(n,\mathbb C)$ be the Lie algebras of all $n\times n$ upper triangular nilpotent matrices over $\mathbb C$. Then $T=\Id:\n(n,\mathbb C)\lon\n(n,\mathbb C)$ is a relative Rota-Baxter operator on the Lie algebra $\n(n,\mathbb C)$ with respect to the representation $(\n(n,\mathbb C);\rho)$, where $\rho:\n(n,\mathbb C)\lon\gl(\n(n,\mathbb C))$ is defined by
  $$
  \rho(A)B=AB,\quad \forall A,B\in \n(n,\mathbb C).
  $$
Thus, the cohomology of the relative Rota-Baxter operator $T$ is the Chevalley-Eilenberg cohomology of $\n(n,\mathbb C)$ with coefficients in the representation $(\n(n,\mathbb C);\rho)$. Since $\n(n,\mathbb C)$ is a nilpotent Lie algebra and the $\n(n,\mathbb C)$-module $\n(n,\mathbb C)$ does contain the trivial $\n(n,\mathbb C)$-module (namely the submodule generated by the matrix with non-zero entry in the upper right corner), we have by Dixmier's Theorem 2 of \cite{Dix} that
$$\dim \huaH_\CE^0(\n(n,\mathbb C),\n(n,\mathbb C))\geq 1,\,\,\,\dim \huaH_\CE^{n(n-1)/2}(\n(n,\mathbb C),\n(n,\mathbb C))\geq 1,$$
and
$$\dim \huaH_\CE^i(\n(n,\mathbb C),\n(n,\mathbb C))\geq 2,\quad\,\,\forall 0<i<\frac{n(n-1)}{2}.$$
}
\end{ex}

\begin{ex}\label{example-8}{\rm(\cite{STZ})} {\rm
Let $(\la,[\cdot,\cdot]_\la)$ be a Leibniz algebra. Then the natural projection $$T=\pr:\la\lon\la_\Lie$$ gives a relative averaging operator on the Lie algebra $\la_\Lie$ with respect to the representation $(\la;\rho)$, where $\rho:\la_\Lie\lon\gl(\la)$ is defined by
  $$
  \rho(\bar{x})y=[x,y]_\la,\quad \forall x,y\in\la.
  $$
Then $(\la,[-,-]_T)$ is the original Leibniz algebra $(\la,[-,-]_\la)$. Moreover, $(\la_\Lie;\varrho^L,\varrho^R)$ is a representation of the Leibniz algebra $(\la,[-,-]_T)$, where $\varrho^L,\varrho^R: \la\lon\gl(\la_\Lie)$ is defined by
$$
\varrho^L(x)\bar{y}=[\bar{x},\bar{y}]_{\la_\Lie},\quad \varrho^R(x)\bar{y}=[\bar{y},\bar{x}]_{\la_\Lie}-\overline{[y,x]_\la}=0.
$$
Thus, the cohomology of the relative  averaging operator $T$ is the Loday-Pirashvili cohomology of $(\la,[\cdot,\cdot]_\la)$ with   coefficients in the antisymmetric representation $(\la_\Lie;\varrho^L,\varrho^R)$.

Let now $(\la,[-,-]_\la)$ be a finite-dimensional semisimple Leibniz algebra over a field of characteristic zero.  By Theorem 4.3 of \cite{FW}, we obtain that $\huaHL^p(\la,\la_\Lie)=\{0\}$ for $p\geq 2$ and an exact sequence
$$0\to (\la_\Lie)_\anti\to \huaHL^0(\la,\la_\Lie)\to \la_\Lie^{\la_\Lie}\to \Hom_\la(\la,\la_\Lie)\to \huaHL^1(\la,\la_\Lie)\to 0.$$
As $\la_\Lie$ is already antisymmetric and $\huaHL^0(\la,\la_\Lie)=\la_\Lie$, the first (non-trivial) map in this sequence is an isomorphism. Furthermore $\la_\Lie^{\la_\Lie}=\{0\}$, because $\la_\Lie$ is semisimple. Therefore the last (non-trivial) map in this sequence is also an isomorphism (as in  Lemma 1.4 of  \cite{FW}).
}
\end{ex}

\vspace{2mm}
\noindent
{\bf Acknowledgements. } This research is supported by NSFC (11922110,12001228).

 \end{document}